\documentclass[11pt,a4paper,twoside]{amsart}
\pdfoutput=1
\usepackage[T1]{fontenc}
\usepackage[english]{babel}

\usepackage{amsthm,amssymb,amsfonts,mathrsfs,amsmath,fancyhdr}

\usepackage[unicode,pdftex, bookmarks, linktocpage=true]{hyperref}
\hypersetup{hypertexnames=false, colorlinks=true, citecolor=newblue, linkcolor=newblue, urlcolor=magenta, pdfpagemode=UseNone , breaklinks=true}

\usepackage{url}

\usepackage[maxbibnames=99,style=numeric, backend=biber, backref, backrefstyle=none, maxalphanames=6]{biblatex}
\usepackage[babel]{csquotes}

\usepackage[usenames,dvipsnames,usenames,dvipsnames,svgnames,table]{xcolor}
\definecolor{newblue} {RGB} {0,100,200}

\usepackage[left=2.5cm, right=2.5cm,top=2.3cm,bottom=2.3cm]{geometry}
\usepackage[all]{xy}
\usepackage{txfonts}
\usepackage{bookmark}

\usepackage{tikz}
\usetikzlibrary{external,matrix,arrows,calc,decorations,patterns,decorations.markings}
\usepackage{tikz-cd}
\usepackage{pgfplots}
\pgfplotsset{compat=1.15}

\usepackage{etoolbox}

\makeatletter
\def\section{\@startsection{section}{1}%
\z@{.7\linespacing\@plus\linespacing}{.5\linespacing}%
{\large \scshape \centering}}
\makeatother

\usepackage{bbm}
\usepackage{tkz-euclide}
\tikzset{
  dotted/.style={pattern=dots,pattern color=#1},
  dotted/.default=black
}
\tikzset{
  fdotted/.style={pattern=crosshatch dots,pattern color=#1},
  fdotted/.default=black
}
\tikzset{
  scopedlines/.style={pattern=north east lines,pattern color=#1},
  scopedlines/.default=black
}
\tikzset{
  hrlines/.style={pattern=horizontal lines,pattern color=#1},
  hrlines/.default=black
}
\newcommand*{\DashedArrow}[1][]{\mathbin{\tikz [baseline=-0.25ex,-latex, dashed,#1] \draw [#1] (0pt,0.5ex) -- (1.3em,0.5ex);}}
\usepackage[expansion=false]{microtype}

\usepackage{mathtools}

\usepackage{dsfont}
\newcommand{\norm}[1]{\left\lVert#1\right\rVert}

\newenvironment{proof1}{\paragraph{\bf{Proof  of Theorem \ref{mainthm1}, item 2:}}}{\hfill$\blacksquare$ \medskip}
\newenvironment{proof2}{\paragraph{\bf{Proof  of Theorem \ref{mainthm1}, item 1:}}}{\hfill$\blacksquare$ \medskip}
\newenvironment{proof3}{\paragraph{\bf{Proof  of Theorem \ref{mainthm3}:}}}{\hfill$\blacksquare$ \medskip}

\expandafter\let\expandafter\oldproof\csname\string\proof\endcsname
\let\oldendproof\endproof
\renewenvironment{proof}[1][\proofname]{%
  \oldproof[\normalfont \bf #1:]%
}{\oldendproof}

\theoremstyle{plain}
\newtheorem{thm}{Theorem}[section]

\theoremstyle{definition}
\newtheorem*{plan}{\normalfont \bf Organization of the paper}

\theoremstyle{plain}
\newtheorem{lem}[thm]{Lemma}

\theoremstyle{definition}
\newtheorem{defn}[thm]{Definition}

\theoremstyle{definition}
\newtheorem{rmk}[thm]{Remark}

\theoremstyle{definition}

\theoremstyle{plain}
\newtheorem{cor}[thm]{Corollary}

\theoremstyle{plain}
\newtheorem{prop}[thm]{Proposition}

\theoremstyle{definition}
\newtheorem{ex}[thm]{Example}

\theoremstyle{definition}

\theoremstyle{definition}
\newtheorem*{histnote}{Historical note}

\theoremstyle{definition}
\newtheorem*{ack}{Acknowledgements}

\title[Boucksom-Zariski and Weyl chambers on IHS manifolds]{Boucksom-Zariski and Weyl chambers on Irreducible Holomorphic Symplectic Manifolds}
\author[Francesco Antonio Denisi]{ Francesco Antonio Denisi}
\address{Institut Elie Cartan de Lorraine, F-54506 Vand\oe{}uvre-l\`{e}s-Nancy Cedex, France}
\email{francesco.denisi@univ-lorraine.fr}

\begin{document}
\begin{abstract}
Inspired by the work of Bauer, K\"uronya, and Szemberg, we provide for the big cone of a projective irreducible holomorphic symplectic (IHS) manifold a decomposition into chambers (which we describe in detail), in each of which the support of the negative part of the divisorial Zariski decomposition is constant. We see how the obtained decomposition of the big cone allows us to describe the volume function. Moreover, similarly to the case of surfaces, we see that the big cone of a projective IHS manifold admits a decomposition into simple Weyl chambers, which we compare to that induced by the divisorial Zariski decomposition. To conclude, we determine the structure of the pseudo-effective cone.
\end{abstract}
\maketitle
\tableofcontents
\section{Introduction}
An irreducible holomorphic symplectic (IHS) manifold is a simply connected compact Kähler manifold carrying (up to multiplication by a scalar) a unique holomorphic 2-form, which is symplectic. They are the higher dimensional analogs of complex K3 surfaces (which are the IHS manifolds of dimension $2$) and one of the three building blocks of compact Kähler manifolds. At present, very few deformation classes of IHS manifolds are known, namely, the K3$^{[n]}$ and $\mathrm{Kum}_n$ deformation classes, for any $n \in \mathbb{N}$ (see \cite{Beau}), and the OG10 and OG6 deformation classes, discovered by O'Grady (see \cite{OGrady1} and \cite{OGrady2} respectively).  In this article, we study the big cone, the pseudo-effective cone, and an asymptotic invariant of linear series of a projective IHS manifold.

As most of this work was inspired by the article \cite{Bau} of Bauer, K\"uronya and Szemberg, we first recall the surface case. An effective $\mathbb{R}$-divisor $D$ on a smooth projective surface $S$ (over an algebraically closed field of arbitrary characteristic), can be written as 
\[
D=P(D)+N(D),
\] 
where :
\begin{enumerate}
\item $P(D)$ is nef and is orthogonal to (any irreducible component of) $N(D)$ (with respect to the usual intersection form),
\item $N(D)$ is effective, so that $N(D)=\sum_ia_iN_i$, with $a_i \geq 0$ for every $i$, and the $N_i$ are the irreducible components of $N(D)$,
\item The Gram matrix $(N_i\cdot N_j)_{i,j}$ of $N(D)$ is negative-definite, if $N(D) \neq 0$.
\end{enumerate} 
 The divisor $P(D)$ is the \textit{positive part} of $D$, while $N(D)$ is the \textit{negative part} of $D$. The decomposition $D=P(D)+N(D)$ with the listed properties is unique and is known as the \textit{Zariski decomposition} of $D$ (for more details see for example \cite{Bauer}). 

\begin{histnote}
The Zariski decomposition was first obtained for effective $\mathbb{Q}$-divisors by Zariski in his fundamental paper \cite{Zariski}.  Then it was extended by Fujita in \cite{Fujita} to the case of pseudo-effective  $\mathbb{Q}$-divisors. The existence and the uniqueness of the Zariski decomposition for effective $\mathbb{R}$-divisors can be easily deduced using the method of Bauer in \cite{Bauer}.
\end{histnote}

In \cite{Bau} the authors described the big cone of $S$ in terms of variation of the Zariski decomposition in it.

\begin{thm}[Th. Bauer, A. K\"uronya, T. Szemberg]\label{decompsurface}

 Let $S$ be a smooth projective surface.
\begin{enumerate}
\item There is a locally finite decomposition of $\mathrm{Big}(S)$ into locally rational polyhedral subcones, such that in each subcone the support of the negative part of the Zariski decomposition of the divisors is constant.
\item Let $D$ be a big $\mathbb{R}$-divisor on $S$, with Zariski decomposition $D=P(D)+N(D)$. We have
\[
\mathrm{vol}(D)=(P(D))^2=(D-N(D))^2.
\]
Therefore on a chamber $\Sigma_P$ on which the support of the negative part is constant, the volume is given by a homogeneous quadratic polynomial.

\end{enumerate}

\end{thm}
The subcones giving the decomposition in Theorem \ref{decompsurface} are called \textit{Zariski chambers} (for the definition and further details we refer the reader to \cite{Bau}).

Thanks to the work \cite{Beau} of Beauville, the cohomology group $H^2(X,\mathbb{C})$ of an IHS manifold $X$ can be endowed with a quadratic form $q_X$, known as Beauville-Bogomolov-Fujiki form. Up to a rescaling, the restriction of $q_X$ to $H^2(X,\mathbb{Z})$ is integral, and its restriction to $\mathrm{Pic}(X) \subset H^2(X,\mathbb{Z})$ shares many properties with the usual intersection form on a surface; below are those that will be mostly used:
\begin{itemize}
\item $q_X(\alpha)>0$ for any Kähler class $\alpha$ on $X$ (cf. 1.10 of \cite{Huy1}).
\item $q_X(\alpha,N)>0$ for any non-zero effective divisor $N$ on $X$ (cf. 1.11 of \cite{Huy1}).
\item $q_X$ is an \textit{intersection product}, namely $q_X(E,E')\geq 0$ for any couple of distinct and non-zero prime divisors on $X$ (cf. Proposition 4.2 of \cite{Bouck}).
\end{itemize}

Also, for an arbitrary compact complex manifold, Boucksom in \cite{Bouck} obtained a sort of "higher dimensional Zariski decomposition", called \textit{divisorial Zariski decomposition} (see also the algebraic approach of Nakayama in \cite{Nakayama}). In particular, in the case of IHS manifolds, Boucksom characterized his decomposition in terms of the Beauville-Bogomolov-Fujiki form. 
\begin{thm}[Theorem 4.8 in \cite{Bouck}]\label{thm1}
Let $D=\sum_{i=1}^ka_i D_i$ be an effective $\mathbb{R}$-divisor on a projective IHS manifold $X$, i.e. $a_i\geq 0$ and $D_i$ is a prime divisor, for each $i=1,\dots,k$. Then $D$ admits a divisorial Zariski decomposition, i.e. we can write uniquely $D=P(D)+N(D)$ such that:
\begin{enumerate}
\item $P(D)$ is an effective, $q_X$-nef $\mathbb{R}$-divisor.
\item $N(D)$ is an effective divisor and, if non-zero, it is $q_X$-exceptional.
\item $P(D)$ is orthogonal (with respect to $q_X$) to each irreducible component of $N(D)$.
\end{enumerate}
The divisor $P(D)$ (resp. $N(D)$) is called the \textit{positive part} (resp. \textit{negative part}) of $D$. 
\end{thm}

 Let us explain the terminology adopted in Theorem \ref{thm1}. A divisor $D$ on a projective IHS manifold $X$ is said $q_X$-nef if $q_X(D,E)\geq 0$ for any prime divisor $E$. An effective $\mathbb{R}$-divisor is said $q_X$-exceptional if the intersection matrix (with respect to $q_X$) of its irreducible components is negative-definite.  In particular, a prime divisor $E$ on $X$ is called $q_X$-exceptional if $q_X(E)<0$.

\begin{rmk}\label{rmkpac}
Following the approach of Kapustka, Mongardi, Pacienza and Pokora in \cite{Pac} to Boucksom's divisorial Zariski decomposition (which in turn relies on the ideas of Bauer contained in \cite{Bauer}) one can see that the positive part $P(D)$ of an effective divisor $D$ on a projective IHS manifold $X$ can be defined as the maximal $q_X$-nef subdivisor of $D$.
\end{rmk}

\begin{rmk} The divisorial Zariski decomposition is rational, namely, under the assumptions of Theorem \ref{thm1}, if the $a_i$ are rational, then $P(D)$ (and so also $N(D)$) is a  $\mathbb{Q}$-divisor.
\end{rmk}

\begin{rmk}\label{rmk3} We note that the divisorial Zariski decomposition behaves well with respect to the numerical equivalence relation. Indeed, let $X$ be a projective IHS manifold and $\alpha$ an effective class in $N^1(X)_{\mathbb{R}}$. Suppose $\alpha = [E_1]$, $\alpha = [E_2]$, where $E_1,E_2$ are effective $\mathbb{R}$-divisors. Write $E_1=P(E_1)+N(E_1)$ (resp. $E_2=P(E_2)+N(E_2)$) for the divisorial Zariski decomposition of $E_1$ (resp. $E_2$). Then $E_1-E_2=T$, where $T$ is numerically trivial. Thus 
\[
P(E_1)+N(E_1)=(P(E_2)+T)+N(E_2),
\] 
and by the unicity of the divisorial Zariski decomposition, we must have $P(E_1)=P(E_2)+T$, because $P(E_2)+T$ is $q_X$-nef, $q_X$-orthogonal to $N(E_1)$ and $N(E_1)$ is $q_X$-exceptional by hypothesis. In particular $P(E_1) \equiv_{\mathrm{num}} P(E_2)$ and $N(E_1)=N(E_2)$, thus $[D]=[P(E_1)]+[N(E_1)]$ in $N^1(X)_{\mathbb{R}}$.
\end{rmk}

In the case of surfaces, the intersection form with its properties and Zariski decomposition played a fundamental role in obtaining Theorem \ref{decompsurface}. On the other hand, on an IHS manifold, we have the Beauville-Bogomolov-Fujiki form and the divisorial Zariski decomposition, and so it seems natural to ask whether Theorem \ref{decompsurface} can be generalized to the case of IHS manifolds. We were able to answer positively to this question. In particular, we obtained the following
theorem, which is one of the main results of this work.

\begin{thm}\label{mainthm1}
Let $X$ be a projective irreducible holomorphic symplectic manifold of complex dimension $2n$.
\begin{enumerate}
\item There is a locally finite decomposition of the big cone $\mathrm{Big}(X)$ of $X$ into locally rational polyhedral subcones, called \textit{Boucksom-Zariski chambers}, such that in each subcone the support of the negative part of the divisorial Zariski decomposition of the divisors is constant.
\item For any big $\mathbb{R}$-divisor $D$ with divisorial Zariski decomposition $D=P(D)+N(D)$, we have $\mathrm{vol}(D)=(P(D))^{2n}$. In particular, on each Boucksom-Zariski chamber, the volume function is expressed by a single homogeneous polynomial of degree $2n$, and so it is piecewise polynomial.
\end{enumerate} 
\end{thm}

\begin{rmk}
The projectivity assumption in Theorem \ref{mainthm1} is not restrictive, because a compact complex manifold having a big line bundle is Moishezon, hence, if K\"ahler, also projective.
\end{rmk}

With the notation of Theorem \ref{mainthm1}, recall that the volume of an integral divisor $D$ on $X$ is defined as 
\[
\mathrm{vol}(D):=\limsup_{k \to \infty} \frac{h^0(X,kD)}{k^{2n}/(2n)!}.
\]
See Subsection 2.1 for the definition in the rational case and the real case.

Now, let us make more precise item 1 in the statement of Theorem \ref{mainthm1}.

\begin{defn}\label{deflocalpolyhed} Let $V$ be a finite dimensional $\mathbb{R}$-vector space. A closed convex cone $K\subset V$ with a non-empty interior is locally rational polyhedral at $v \in K$ if $v$ has an open neighborhood $U=U(v)$, such that $K \cap U$ is defined in $U$ by a finite number of rational linear inequalities. 
\end{defn}

\begin{rmk}
In the statement of Theorem \ref{mainthm1}, item 1,
\begin{enumerate} 
\item "locally finite" means that any point in $\mathrm{Big}(X)$ has a neighbourhood meeting only a finite number of chambers;
\item  "locally rational polyhedral subcones" means that any of these subcones is locally rational polyhedral at any big point of its closure in the N\'{e}ron-Severi space, in the sense of Definition \ref{deflocalpolyhed}.
\end{enumerate}
\end{rmk}
 The strategy of our proof is based on that adopted by Bauer, K\"uronya and Szemberg in \cite{Bau}.
   \begin{rmk}
  Note that, in \cite{Bau}, the authors also studied the behavior of the augmented base loci (of divisors on a smooth projective surface). In this case, the problem led to another decomposition of the big cone into chambers, called \textit{stability chambers}, on each of which the augmented base loci remain constant. It turned out that the stability chambers essentially agree with the Zariski chambers (see Theorem \ref{stabilitychambers}). In general, in the case of a projective IHS manifold, the Boucksom-Zariski chambers do not agree with the stability chambers (see Remark \ref{rmkstabchambers}).
  \end{rmk}
  
  The big cone of a smooth projective surface $S$ can also be decomposed into \textit{simple Weyl chambers}, which are defined as the connected components of 
\[
\mathrm{Big}(S) \setminus \left(\bigcup_{C \in \mathrm{Neg}(S)} C^{\perp} \right),
\]
 where $\mathrm{Neg}(S)$ is the set of irreducible and reduced curves of negative square and, given a curve $C$, $C^{\perp}$ is the hyperplane in the N\'{e}ron-Severi space of classes orthogonal to the curve $C$, with respect to the usual intersection form.  In several papers, the Zariski chambers and the simple Weyl chambers have been compared.  First, Bauer and Funke in \cite{Bauer1} studied the case of projective K3 surfaces. Then Rams and Szemberg in \cite{Szemberg}, and, more recently, Hanumanthu and Ray in \cite{Hanumanthu}, studied the case of any smooth projective surface. We are mainly interested in the following result.

 \begin{thm}[Theorem 3 of \cite{Szemberg}, see also Theorem 3.3 of \cite{Hanumanthu}]
 Let $S$ be a smooth projective surface. The following conditions are equivalent:
 \begin{enumerate}
 \item The interior of the Zariski chambers on $S$ coincides with the simple Weyl chambers.
 \item If two different reduced and irreducible negative curves $C_1, C_2$ on $S$ intersect properly (i.e. $C_1\cdot C_2>0$), then $C_1\cdot C_2\geq \sqrt{C_1^2C_2^2}$.
 \end{enumerate}
  \end{thm} 
The notion of simple Weyl chamber can be naturally generalized to the case of a projective IHS manifold $X$, by replacing the reduced and irreducible curves with the $q_X$-exceptional prime divisors. In this way, as for surfaces, we obtain another decomposition of $\mathrm{Big}(X)$ into chambers, which we also call \textit{simple Weyl chambers}. As for surfaces, it is natural to compare the two decompositions of $\mathrm{Big}(X)$ introduced so far. The main result of this part of the work is the following Theorem, which is a natural generalization of the above result by Rams and Szemberg.

 \begin{thm}\label{mainthm3}
 Let $X$ be a projective IHS manifold. The following conditions are equivalent:
 \begin{enumerate}
 \item The interior of the Boucksom-Zariski chambers coincides with the simple Weyl chambers.
 \item If two different $q_X$-exceptional prime divisors $D_1, D_2$ on $X$ intersect properly (i.e. $q_X(D_1,D_2)>0$), then $q_X(D_1,D_2)\geq \sqrt{q_X(D_1)q_X(D_2)}$.
 \end{enumerate}
  \end{thm}
   Theorem \ref{mainthm3} explains when the decomposition of $\mathrm{Big}(X)$ into Boucksom-Zariski chambers and the one into simple Weyl chambers coincide, and, to obtain it, we first generalize Theorem 3.1 and 3.2 of \cite{Hanumanthu} (which are in turn a generalization of the items in Theorem 2 of \cite{Bauer1}) to our case (and these are interesting results in their own); this will be enough to deduce Theorem \ref{mainthm3}. 
   
 \begin{plan}
 Let $X$ be a projective IHS manifold.
 \begin{itemize}
 \item In the second section, we give the basic notions and results that will be used to achieve our goals.
 \item In the third section, we extend Theorem 3.6 contained in \cite{Pac} to the case of effective $\mathbb{R}$-divisors on $X$. 
 \item In the fourth section, we determine the structure of the pseudo-effective cone $\overline{\mathrm{Eff}(X)}$. It turns out that, thanks to the properties of the Beauville-Bogomolov-Fujiki form, the structure of $\overline{\mathrm{Eff}(X)}$ is analogous to that of the Kleiman-Mori cone in the case of surfaces.
 \item In the fifth section, we give a decomposition of the big cone of $X$ into chambers (that will be called \textit{Boucksom-Zariski chambers}), by studying how the divisorial Zariski decomposition varies in $\mathrm{Big}(X)$, and describe these chambers in detail.
 \item In the sixth section, using the above decomposition of $\mathrm{Big}(X)$, we study the behavior of the volume function $\mathrm{vol}(-)$ on $\mathrm{Big}(X)$. In particular, we show that it is locally polynomial, and the pieces of the big cone at which $\mathrm{vol}(-)$ is polynomial are exactly the Boucksom-Zariski chambers. The description of $\mathrm{vol}(-)$ is completely analogous to that in the case of surfaces. While in the case of surfaces, the result can be easily deduced from the description of the Zariski chambers, in the case of projective IHS manifolds we have to pay more attention.
 \item In the seventh section we describe the Boucksom-Zariski chambers and the volume function of the Hilbert square of a certain K3 surface. Also, we show that, in general, the augmented base loci are not constant in the interior of the Boucksom-Zariski chambers.
 \item In the eighth section, after introducing and characterizing the simple Weyl chambers for a projective IHS manifold, we compare the decomposition of $\mathrm{Big}(X)$ into Boucksom-Zariski chambers to the one into simple Weyl chambers.
 \end{itemize}
 \end{plan}

\section{Background, notations and conventions}
Let us start by recalling what is an irreducible holomorphic symplectic manifold.
\begin{defn}\label{defIHS}
An irreducible holomorphic symplectic manifold is a simply connected compact K\"ahler manifold $Y$, such that $H^0(Y,\Omega^2_Y)=\mathbb{C}\sigma$, where $\sigma$ is a holomorphic symplectic form.
\end{defn}
For a general introduction to the subject, we refer the reader to \cite{Joyce}. We will work in the projective setting, thus in what follows we assume that $X$ is a projective IHS manifold of complex dimension $2n$.  A prime divisor on $X$ will be a reduced and irreducible hypersurface. We said that $H^2(X,\mathbb{C})$ can be endowed with a quadratic form $q_X$, known as the Beauville-Bogomolov-Fujiki form. In particular, choosing the symplectic form $\sigma$ in such a way that  $\int_X(\sigma\overline{\sigma})^n=1$, one can define
\[
q_X(\alpha):=\frac{n}{2}\int_X (\sigma\overline{\sigma})^{n-1}\alpha^2+(1-n)\left(\int_X\sigma^n\overline{\sigma}^{n-1}\alpha\right)\cdot \left(\int_X\sigma^{n-1}\overline{\sigma}^n\alpha\right),
\]
for any $\alpha \in H^2(X,\mathbb{C})$. The quadratic form $q_X$ is non-degenerate and its signature on $H^2(X,\mathbb{R})$ is $\left(3,b_2(X)-3\right)$. Also, the following important fact holds.
\begin{thm}[Beauville-Bogomolov-Fujiki relation]\label{fujikirelation}
There exists a positive constant $c_X \in \mathbb{R}^{>0}$, known as Fujiki constant, such that $q_X(\alpha)^n= c_X\int_X \alpha^{2n}$,
for all $\alpha \in H^2(X,\mathbb{C})$. In particular, $q_X$ can be renormalized such that $q_X$ is a
primitive integral quadratic form on $H^2(X, \mathbb{Z})$. 
\end{thm}
\begin{proof}
See for example Proposition 23.14 of \cite{Joyce}, or \cite{Fujiki}.
\end{proof}

 \subsection{\textit{Positivity notions}}
 Let $Y$ be a (complex) normal projective variety of (complex) dimension $d$. Recall that the Néron-Severi group of $Y$ can be defined as $N^1(Y) \otimes_{\mathbb{Z}} \mathbb{R}$, where $N^1(Y):=\frac{\mathrm{Div}(Y)}{\equiv_{\mathrm{num}}}$, $\mathrm{Div}(Y)$ is the group of integral Cartier divisors of $Y$, and $\equiv_{\mathrm{num}}$ is the numerical equivalence relation. A divisor will be a formal linear combination (with coefficients in $\mathbb{R}$) of integral Cartier divisors. We tacitly assume that any divisor is Cartier.  
 \begin{rmk}
 In the case of IHS manifolds the numerical and linear equivalence relations coincide, hence we can also define the Néron-Severi space as $N^1(X)_{\mathbb{R}}:=\mathrm{Pic}(X) \otimes_{\mathbb{Z}} \mathbb{R}$. Also, it is important to keep in mind that since 
\[
N^1(X)_{\mathbb{R}} \cong H^{1,1}(X,\mathbb{Z}) \otimes_{\mathbb{Z}} \mathbb{R} \subset H^{1,1}(X,\mathbb{R}),
\] 
the Beauville-Bogomolov-Fujiki form respects the numerical equivalence relation.
 \end{rmk}
 \begin{defn}
 An integral divisor $D$ on $Y$ is said big if there exists a positive integer $m$ such that $mD \equiv_{\mathrm{num}} A+N$, where $A$ is integral and ample, and $N$ is integral and effective. When $D$ is a $\mathbb{Q}$-divisor, we say that it is big if there exists a positive integer $m$ such that $mD$ is integral and big. If $D$ is an $\mathbb{R}$-divisor, we say that it is big if $D=\sum_ia_iD_i$, where $a_i$ is a positive real number and $D_i$ is integral and big, for any $i$.
 \end{defn}
 
 The latter tells us that, up to being sufficiently twisted, a big, integral divisor induces a rational map that is birational onto its image. Hence, in a certain sense, being big means being "generically ample". Note that bigness behaves well with respect to $\equiv_{\mathrm{num}}$, namely, if $D\equiv_{\mathrm{num}} D'$, $D$ is big if and only if $D'$ is big.
 
 \begin{ex}
 Any ample divisor on $Y$ is big. On the other hand, a big divisor is not necessarily ample. For example, consider $\pi \colon \mathrm{Bl}_P\mathbb{P}^2 \to \mathbb{P}^2$, and let $l$ be a line of $\mathbb{P}^2$ not containing the closed point $P$. Then $\pi ^{*}l$ is big but not ample because $\pi ^{*}l$ does not intersect the exceptional divisor. An example of a non-big divisor is given by a $q_X$-exceptional prime divisor $E$ on $X$ (see item 3 in Remark \ref{rmk6}). For instance, a smooth rational curve on a smooth K3 surface is not big.
 \end{ex}
 
 We will denote the ample cone of $Y$ by $\mathrm{Amp}(Y)$.  Recall that the elements of $\mathrm{Amp}(Y)$ are those of the form $\alpha=\sum_ia_i[D_i]$, where any $a_i$ is a positive, real number, and any $D_i$ is an integral, ample divisor. 
 
 A divisor $D$ on $Y$ is said nef if $D \cdot C \geq 0$, for any irreducible and reduced curve $C$ on $Y$. The nef cone $\mathrm{Nef}(Y)$ is the convex cone spanned by the nef divisors' classes. It is known that $\mathrm{Nef}(Y)=\overline{\mathrm{Amp}(Y)}$, and that $\mathrm{Amp}(Y)$ is the interior of  $\mathrm{Nef}(Y)$ (see Theorem 1.4.23 of \cite{Laz}).
 
  The big cone $\mathrm{Big}(Y)$ of $Y$ is the convex cone of all big $\mathbb{R}$-divisor classes. The pseudo-effective cone of $Y$ can be defined as the closure of the big cone in $N^1(Y)_{\mathbb{R}}$ and is denoted by $\overline{\mathrm{Eff}(Y)}$. By a standard result (see Theorem 2.2.26 of\cite{Laz}), the big cone is the interior of $\overline{\mathrm{Eff}(Y)}$. 
\begin{defn} 
 Let $D$ be an integral divisor on $Y$. The volume of $D$ is defined as
\[
\mathrm{vol}(D):=\limsup_{k \to \infty} \frac{h^0(Y,kD)}{k^d/d!}.
\]
\end{defn}
It is known that $\mathrm{vol}(D)>0$ if and only if $D$ is big. One defines the volume of a rational Cartier divisor $D$ by picking an integer $k$ such that $kD$ is integral, and by setting $\mathrm{vol}(D):=\frac{1}{k^{m}}\mathrm{vol}(kD)$. This definition does not depend on the integer $k$ we have chosen (see Lemma 2.2.38 on \cite{Laz}). One can extend the notion of volume to every $\mathbb{R}$-divisor. In particular, pick $D \in \mathrm{Div}_{\mathbb{R}}(Y)=\mathrm{Div}(Y)\otimes_{\mathbb{Z}} \mathbb{R}$ and let $\{D_k\}_k$ be a sequence of $\mathbb{Q}$-divisors converging to $D$ in $N^1(Y)_{\mathbb{R}}$. We define
\[
\mathrm{vol}(D):=\lim_{k \to \infty}\mathrm{vol}(D_k).
\]
This number is independent of the choice of the sequence $\{D_k\}_k$ (see Theorem 2.2.44 of \cite{Laz}). Moreover, two numerically equivalent divisors have the same volume (see Proposition 2.2.41 of \cite{Laz}), and so the volume can be safely studied in $N^1(Y)_{\mathbb{R}}$, and by Theorem 2.2.44 of \cite{Laz} it defines a continuous function $\mathrm{vol}(-) \colon N^1(Y)_{\mathbb{R}} \to \mathbb{R}$, called \textit{volume function}.

\subsection{\textit{Cones associated with IHS manifolds}}
One of the main characters of this paper will be the following cone in $N^1(X)_{\mathbb{R}}$.
\begin{defn}
The $q_X$-nef cone of $X$ is
\[
\mathrm{Nef}_{q_X}(X):=\{\alpha \in N^1(X)_{\mathbb{R}} \; | \; q_X(\alpha,D')\geq 0 \mathrm{\; for\; each \; prime \; divisor \; } D'\}.
\]
\end{defn}
This is a closed cone in $N^1(X)_{\mathbb{R}}$ as it is by definition the intersection of all the closed half-spaces of the form
\[
D^{\geq 0}:=\{\alpha \in N^1(X)_{\mathbb{R}} \; | \; q_X(\alpha,D)\geq 0\},
\]
where $D$ is a prime divisor on $X$.  Moreover, it is non-empty, because it contains $0$.
Given a divisor $D$ on $X$, we will frequently need to consider in $N^1(X)_{\mathbb{R}}$ the hyperplane of classes that are $q_X$-orthogonal to $D$, i.e.
\[
D^{\perp}:=\left\{\alpha \in N^1(X)_{\mathbb{R}} \; | \; q_X(\alpha,D)=0\right\}.
\]
Note that $D^{\perp}$ is really an hyperplane, because $q_X$ is non-degenerate on $N^1(X)_{\mathbb{R}}$.

We recall that the positive cone $\mathscr{C}_X \subset H^{1,1}(X,\mathbb{R}):=H^{1,1}(X,\mathbb{C})\cap H^2(X,\mathbb{R})$ is the connected component of the set $\{\alpha \in H^{1,1}(X,\mathbb{R}) \; | \; q_X(\alpha)>0\}$ containing the K\"ahler cone $\mathscr{K}_X$. 
\begin{defn}
The fundamental exceptional chamber of the positive cone $\mathscr{C}_X$ is defined as
\[
\mathscr{FE}_X:= \left\{\alpha \in \mathscr{C}_X \; | \; q_X(\alpha,[D])>0 \mathrm{\;  for \; every \;} q_X\mathrm{-exceptional} \mathrm{\; prime\; divisor\;} D \right\}.
\]
\end{defn}

In the above definition, the adjective "exceptional" comes from Definition 5.10 in \cite{Mark}.

\begin{rmk}\label{qnefrmk1}
Note that $\mathrm{Nef}_{q_X}(X)=\overline{\mathscr{FE}_X} \cap N^1(X)_{\mathbb{R}}$. Indeed 
\[
\overline{\mathscr{FE}_X}=\left\{\alpha \in \mathscr{C}_X \;|\; q_X(\alpha,E)\geq 0 \text{ for any $q_X$-exceptional prime divisor }E\right\},
\]
and $q_X$ is an intersection product. Hence we are done.
\end{rmk}

Let $f \colon X \DashedArrow[->,densely dashed    ] X'$ be a birational map, where $X'$ is another projective IHS manifold. Recall that $f$ restricts to an isomorphism $f \colon U \to U'$, where $\mathrm{codim}_{X}(X \setminus U)$, $\mathrm{codim}_{X'}(X' \setminus U')\geq 2$ and $X\setminus U$, $X' \setminus U'$ are analytic subsets of $X$ and $X'$ respectively (see for example paragraph 4.4 in \cite{Huy1}). Using the long exact sequence in cohomology with compact support and Poincaré duality, one has the usual chain of isomorphisms
\[
H^2(X, \mathbb{R}) \cong H^2(U, \mathbb{R}) \cong H^2(U', \mathbb{R}) \cong H^2(X',\mathbb{R}),
\]
and the composition is an isometry with respect to $q_X$ and $q_{X'}$ (see for example Proposition I.6.2 of \cite{OGrady} for a proof), hence its restriction to $H^{1,1}(X, \mathbb{R})$ induces an isometry $H^{1,1}(X, \mathbb{R}) \cong H^{1,1}(X', \mathbb{R})$, which we will denote by $f_{*}$ (push-forward). We will denote the inverse of $f_{*}$ by $f^{*}$ (pull-back). Note that $f^{*}$ behaves like the pull-back of line bundles via morphisms in the usual sense because $X$ and $X'$ are isomorphic in codimension one.
\begin{defn}
The birational K\"ahler cone $\mathscr{BK}_X$ is defined as
\[
\mathscr{BK}_X=\bigcup_{f} f^{*}\mathscr{K}_{X'},
\]
where $f$ varies among all the birational maps $f \colon X \DashedArrow[->,densely dashed    ] X'$, where $X'$ is another projective IHS manifold.
\end{defn}
By Proposition 5.6 in \cite{Mark} the closure of $\mathscr{FE}_X$ and $\mathscr{BK}_X$ in $\mathscr{C}_X$ coincide. Hence also $\overline{\mathscr{FE}_X}=\overline{\mathscr{BK}_X}$, where both the closures are taken in $H^{1,1}(X,\mathbb{R})$.

We need to introduce another cone in the N\'{e}ron-Severi space. 
\begin{defn}
\begin{enumerate}
\item  A movable line bundle on $X$ is a line bundle $L$ such that the linear series $|mL|$ has no divisorial components in its base locus, for $m$ sufficiently large and divisible. A divisor $D$ is movable if $\mathscr{O}_X(D)$ is.
\item We define the movable cone $\mathrm{Mov}(X)$ in $N^1(X)_{\mathbb{R}}$ as the convex cone generated by the classes of movable divisors. 
\end{enumerate} 
\end{defn}

\begin{rmk}\label{movableconeisqnefcone}
It is known that 
\[
\overline{\mathrm{Mov}(X)}=\overline{\mathscr{BK}_X}\cap N^1(X)_{\mathbb{R}}=\overline{\mathscr{FE}_X}\cap N^1(X)_{\mathbb{R}}
\] 
(see Theorem 7 and Remark 9 in \cite{Hassett}), where the closure of $\mathrm{Mov}(X)$ is taken in $N^1(X)_{\mathbb{R}}$.
\end{rmk}

Now we give a useful property of big and $q_X$-nef divisors on a projective IHS manifold.

\begin{lem}\label{bigqnefispositive}
If $D$ is a big and $q_X$-nef integral divisor, then $q_X(D)>0$. 
\begin{proof}
Since $D$ is big, we can write $D\equiv A+N$, with $A$ ample and $N$ effective. Then $q_X(D)=q_X(D,A)+q_X(D,N)$. Now, $q_X(D,A)>0$, because the class of $D$ in $N^1(X)_{\mathbf{R}}$ is effective, and $q_X(D,N)\geq 0$, because $D$ is $q_X$-nef. This concludes the proof.
\end{proof}
\end{lem}
 
 \subsection{\textit{A Hodge index-type Theorem for IHS manifolds}}
 In what follows we provide a version of the \textit{strong Hodge-index Theorem} for projective IHS manifolds, as we did not find it in the literature and it plays an important role in the present paper. We start with the following lemma of linear algebra.
\begin{lem}\label{lemsignature}
Let $q$ be a quadratic form of signature $(1,\rho -1)$ on a finite dimensional, real vector space $V$, with $\mathrm{dim}(V)= \rho$. Suppose that $\pi$ is a $2$-plane containing a positive direction. Then the signature of $q_{| \pi}$ is $(1,1)$.
\begin{proof}
Since $\pi$ contains a positive direction, the signature of $q_{|\pi}$ is $(1,1)$ or $(1,0)$ (so in the second case $q_{|\pi}$ is degenerate). If $q_{|\pi}$ is degenerate then $\pi$ contains no negative directions. But this is absurd because $\pi$ intersects all the hyperplanes on which $q_X$ is negative-definite (and we have one by hypothesis). Thus the signature is $(1,1)$.
\end{proof}
\end{lem}

\begin{prop}[Hodge-type inequality]\label{propsignature}
Let $D,D'$ be two integral divisors on $X$, such that $q_X([D])>0$. Then
\begin{enumerate}
\item $(q_X(D,D'))^2\geq q_X(D)q_X(D')$,
\item Equality in the previous item holds if and only if $D$ and $D'$ are linearly dependent in $N^1(X)_{\mathbb{Q}}$, i.e. if and only if $rD \equiv_{\mathrm{num}}  D'$, for some $r \in \mathbb{Q}$, and so if and only if $aD \equiv_{\mathrm{num}} bD'$ for some $a,b \in \mathbb{Z}$.
\end{enumerate}
\end{prop}
\begin{proof}
If $[D']=0$ there is nothing to prove. So we can assume $[D'] \neq 0$.
\begin{enumerate}\item Consider the subspace $\pi:= \left<[D],[D']\right> \subset N^1(X)_{\mathbb{R}}$. 
The determinant of the matrix representing $q_X$ restricted to $\pi$ is $d:=q_X(D)q_X(D')-q_X(D,D')^2$ and hence the signature of $(q_X)_{|\pi}$ is $(1,1)$ if and only if $d<0$, i.e. $(q_X(D,D'))^2> q_X(D)q_X(D')$. The equality $d=0$ holds if and only if $\mathrm{dim}(\pi)=1$, and we are done.
\item By the previous item we have $(q_X(D,D'))^2= q_X(D)q_X(D')$ if and only if $[D]$ and $[D']$ are linearly dependent in $N^1(X)_{\mathbb{R}}$. But $[D],[D'] \in \mathrm{Pic}(X)$, thus they must be linearly dependent in $N^1(X)_{\mathbb{Q}}$. It follows that $r[D]= [D']$ for some $0 \neq r \in \mathbb{Q}$, i.e. $a[D] = b[D']$ for some $a,b \in \mathbb{Z} \setminus \{0\}$.
\end{enumerate}
\end{proof}

\begin{cor}\label{stronghodge}
Let $D,D'$ be two integral divisors, and assume $q_X(D)>0$. Then, if $q_X(D,D')=0$, we have $q_X(D')<0$ or $[D'] = 0$.
\end{cor}
\begin{proof}
By item $1$ of Proposition \ref{propsignature}, we have
\[
0=(q_X([D],[D']))^2 \geq q_X(D)q_X(D').
\]
Hence, since $q_X([D])>0$, we have $q_X(D') \leq 0$. By contradiction, suppose that $q_X([D'])=0$, but $[D']\neq 0$. As in this case
\[
(q_X([D],[D']))^2=0=q_X([D])q_X([D']),
\] 
by item 2 of Proposition \ref{propsignature}, we have $b[D] = a[D']$, for some $a,b \in \mathbb{Z} \setminus \{0\}$.  Thus
\[
b^2q_X([D])=a^2q_X([D'])=0,
\]
and so it must hold $b=0$. Then $a[D'] = 0$, and since $N^1(X)$ is torsion free we must have $[D']=0$, which is clearly a contradiction.
\end{proof}

\begin{cor}[Strong Hodge-index Theorem]\label{stronghodgeindex}
Let $D$ be an integral divisor on $X$, with $q_X(D)>0$. Consider the decomposition $N^1(X)_{\mathbb{R}}=\mathbb{R}[D]\oplus D^{\perp}$. Then $q_X$ is negative-definite on the hyperplane $[D]^{\perp}$.
\end{cor}
\begin{proof}
By Corollary \ref{stronghodge} we have that $q_X(D')<0$ for every non-zero rational class  $[D'] \in D^{\perp}$. It remains to be proved that the same is true for each non-zero real class $\alpha$ $q_X$-orthogonal to $D$. Any basis $\mathcal{B}$ for $H^{1,1}(X,\mathbb{Z})$ naturally gives an integral basis $\mathcal{B}'$ of $D^{\perp}$. The $\mathbb{Q}$-vector space $D^{\perp}_{\mathbb{Q}}$ spanned by $\mathcal{B}'$ in $D^{\perp}$ is the set rational points of $D^{\perp}$. But $q_X$ is negative-definite on $D^{\perp}_{\mathbb{Q}}$, and $\mathcal{B}'$ is also a basis for $D^{\perp}$, hence $q_X$ is negative-definite on $D^{\perp}$, and this concludes the proof.
\end{proof}

\section{A note on the pseudo-effective cone}
In this section, we determine the structure of the pseudo-effective cone $\overline{\mathrm{Eff}(X)}$. 

\begin{rmk}\label{rmk6} We observe the following:
\begin{enumerate}
\item If $D$ is an effective and big $\mathbb{R}$-divisor on $X$, then $P(D) >0$. Indeed $D \equiv_{\mathrm{num}} A+N$, for some $A$ ample and $N$ effective. Then $P(D) \geq A >0$, as $P(D)$ is the maximal $q_X$-nef subdivisor of $D$ (see Remark \ref{rmkpac}). It follows that a $q_X$-exceptional, effective $\mathbb{R}$-divisor cannot be big and that the class of a $q_X$-exceptional prime divisor $D'$ belongs to the boundary of $\overline{\mathrm{Eff}(X)}$.

\item If $E$ is a prime exceptional divisor, then $nE$ is fixed (i.e. $h^0(nE)=1$) for every $n>0$. Indeed, if $nE$ were not fixed, there would be an effective divisor $D$ linearly equivalent to $mE$ and not containing $E$ in its support, for some $0<m\leq n$. Then we would have $0\leq mq_X(D,E)=m^2q_X(E,E)<0$, which is absurd.
\item Let $D$ be a prime divisor with $q_X(D)=0$, then $[D] \in \partial \overline{\mathrm{Eff}(X)}$. Indeed, as $q_X$ is an intersection product, $D$ is $q_X$-nef. If it were also big, then we would have $q_X(D)>0$ by Lemma \ref{bigqnefispositive}. Thus $[D]$ lies on the boundary of $\overline{\mathrm{Eff}(X)}$.
\end{enumerate}
\end{rmk}
Kleiman's criterion for amplitude on a projective variety inspired the proof of the following proposition, which is interesting on its own.
\begin{prop}\label{lemample}
Let $A$ be an ample divisor on $X$. Then  the linear functional 
\[
q_X(-,A) \colon N^1(X)_{\mathbb{R}} \to \mathbb{R}, \, \mathrm{s.t.} \; q_X(-,A)(\alpha)=q_X(\alpha,A),
\] 
is strictly positive on $\overline{\mathrm{Eff}(X)} \setminus \{0\}$. In particular $\overline{\mathrm{Eff}(X)}$ does not contain any line.
\end{prop}
\begin{proof}
As $[A]$ is K\"ahler, $q_X(-,A)$ is strictly positive on $\mathrm{Eff}(X) \setminus \{0\}$, hence it is nonnegative on $\overline{\mathrm{Eff}(X)}$. Now, assume that $q_X(\alpha,A)=0$ for some pseudo-effective class $\alpha \neq 0$. We note that there exists an integral divisor $D$ such that $q_X(\alpha,D)<0$. Indeed $q_X$ is non-degenerate, and so we can find a class $\beta$ satisfying $q_X(\alpha,\beta)\neq 0$. Without loss of generality, we can assume $q_X(\alpha,\beta)< 0$. If $\mathrm{Pic}(X) \subset \alpha^{\geq 0}$, then $N^1(X)_{\mathbb{R}}= \alpha^{\geq 0}$, which is clearly absurd. Now observe that $D+nA$ is ample for $n\gg 1$, hence we have
\[
0\leq q_X(\alpha,D+nA)=q_X(\alpha,D)<0,
\]
which is a contradiction. Now we show that $\overline{\mathrm{Eff}(X)}$ does not contain lines. Let $l \subset \overline{\mathrm{Eff}(X)}$ be a line. We can assume that $l$ passes through the origin. Let $0\neq \alpha \in l$ be an element. Then $-\alpha$ is also pseudo-effective and so $q_X(t \alpha+(1-t)(-\alpha),A)>0$ for each $t \in [0,1]$. This is absurd, as $0$ belongs to the segment joining $\alpha$ and $-\alpha$.
\end{proof}
 Note that $\overline{\mathrm{Eff}(Y)}$ does not contain lines for any projective variety $Y$. If the variety is smooth, the mentioned result is a particular case of Proposition 1.3 from \cite{Favre}. For a proof which works also in the non-smooth case, see Lemma 4.6 of \cite{Mustata}.

\begin{rmk}\label{rmkpositiveconeinbigcone}
The following inclusion holds true 
\begin{equation}\tag{$\diamond$}\label{equationdiamond}
\mathscr{C}_X\cap N^1(X)_{\mathbb{R}} \subset \mathrm{Big}(X).
\end{equation}
 Indeed, by Proposition 3.8 of \cite{Huy1} and Proposition 1.4 of \cite{Dem}, we have that
 \begin{equation}\tag{\#}\label{eqposcone}
 \mathscr{C}_X\cap N^1(X)_{\mathbb{R}} \subset \overline{\mathrm{Eff}(X)}.
 \end{equation}
 Taking the interiors (in $N^1(X)_{\mathbb{R}}$) of both the members of (\ref{eqposcone}) we obtain (\ref{equationdiamond}).
\end{rmk}

Using again Proposition 1.4 of \cite{Dem}, and Theorem 3.19 of \cite{Bouck}, one obtains the following.

\begin{cor}\label{pseflocpol}
The pseudo-effective cone $\overline{\mathrm{Eff}(X)}$ is locally polyhedral far from $\partial \mathrm{Nef}_{q_X}(X) \cap \partial \overline{\mathrm{Eff}(X)}$, with extremal rays spanned by (the classes of) $q_X$-exceptional prime divisors. 
\end{cor}

 Corollary \ref{pseflocpol} implies that if $q_X(\alpha)<0$, with $\alpha \in \overline{\mathrm{Eff}(X)}$, then $\mathbb{R}^{ \geq 0} \alpha$ is an extremal ray of $\overline{\mathrm{Eff}(X)}$ if and only if $\mathbb{R}^{\geq 0} \alpha = \mathbb{R}^{\geq 0} [D]$, for some $q_X$-exceptional prime divisor $D$. 

The next result was inspired by Theorem 4.13 of \cite{Kollar}.

\begin{cor}\label{structurepseff}
We have
\[
\overline{\mathrm{Eff}(X)}=\overline{\mathscr{C}_X}\cap N^1(X)_{\mathbb{R}}+\sum_D \mathbb{R}^{\geq 0}[D],
\]
where the sum is over all the $q_X$-exceptional prime divisors.
\begin{proof}
The inclusion $\supseteq$ is obvious. It remains to show $\subseteq$, and it is sufficient to prove that every extremal ray of $\overline{\mathrm{Eff}(X)}$ is contained in the right-hand side of the equality in the statement (indeed $\overline{\mathrm{Eff}(X)}$ does not contain lines, thus by Minkowsky's Theorem it is spanned by its extremal rays).  Let $\mathbb{R}^{\geq 0}\alpha$ be an extremal ray. If $q_X(\alpha)\geq 0$, then $\alpha \in \overline{\mathscr{C}_X}$. If $q_X(\alpha)<0$, by Corollary \ref{pseflocpol} there exists a $q_X$-exceptional prime divisor $D$ such that $\mathbb{R}^{\geq 0}\alpha=\mathbb{R}^{\geq 0}[D]$, hence we are done.
\end{proof}
\end{cor}

 Note that if $X$ has Picard number $\rho(X) \geq 2$ and $\beta \in \overline{\mathscr{C}_X}$ spans an extremal ray of $\overline{\mathrm{Eff}(X)}$, then $q_X(\beta)=0$. Also, notice that $q_X(\beta)=0$ does not necessarily imply that $\beta$ spans an extremal ray of $\overline{\mathrm{Eff}(X)}$ (see Example \ref{exchambers}).

\section{Boucksom-Zariski chambers}
The goal of this section is to prove item $1$ of Theorem \ref{mainthm1}. We start with the following useful result.

\begin{lem}[item (iii) of Proposition 3.8 in \cite{Bouck}]\label{positivepartisbig}
Let $D$ be an effective and big $\mathbb{R}$-divisor on $X$, and $D=P(D)+N(D)$ its divisorial Zariski decomposition. Then $P(D)$ is a big and $q_X$-nef divisor.
\end{lem}

\begin{defn}
Given a big $\mathbb{R}$-divisor $D$ on $X$ we define the \textit{$q_X$-null locus} of $D$ as
\[
\mathrm{Null}_{q_X}(D):=\{D' \; | \; D' \mathrm{\;is \;a \;prime\; divisor\; and\; } q_X(D,D')=0\}.
\]
\end{defn}
\begin{rmk}\label{nulllocusfinite}
We note that the $q_X$-null locus of a big $\mathbb{R}$-divisor $D$ is finite. Indeed, let $D'$ be a prime divisor lying in $\mathrm{Null}_{q_X}(D)$. As $D$ is big, we have $[D]=[A]+[N]$ in $N^1(X)_{\mathbb{R}}$, where $A$ is ample and $N$ is effective. As $q_X(A,D')>0$, then $D'$ is an irreducible component of $N$. It immediately follows that the cardinality of $\mathrm{Null}_{q_X}(D)$ is bounded by that of the irreducible components of $N$, and so is finite, because $N$ is a divisor.
\end{rmk}
\begin{defn}
Let $D$ be an effective $\mathbb{R}$-divisor on $X$ and $D=P(D)+N(D)$ its divisorial Zariski decomposition. We define the \textit{$q_X$-negative locus} of $D$ as
\[
\mathrm{Neg}_{q_X}(D):=\{D' \; | \; D' \mathrm{ \;is\; an\; irreducible\; component\; of\; } N(D)\}.
\]
\end{defn}
\begin{rmk}\label{rmk5}
We note that 
\begin{enumerate}
\item given a big  class $\alpha$ and a big  $\mathbb{R}$-divisor $D$ such that $\alpha=[D]$, we can define
\[
\mathrm{Null}_{q_X}(\alpha):=\mathrm{Null}_{q_X}(D),
\] 
because  $q_X$ respects the numerical equivalence relation and so this set does not depend on the representative of $\alpha$ we have chosen,
\item given an effective class $\alpha$ and $D \in \mathrm{Div}_{\mathbb{R}}(X)$ an effective representative of $\alpha$, we can define 
\[
\mathrm{Neg}_{q_X}(\alpha):=\mathrm{Neg}_{q_X}(D),
\] 
because from Remark \ref{rmk3} also this set does not depend on the effective representative of $\alpha$ we have chosen.
\end{enumerate}
\end{rmk}
\begin{rmk}
Let  $D$ be a big $\mathbb{R}$-divisor on $X$, then $\mathrm{Null}_{q_X}(D) \subset \mathrm{Neg}(X)$, where $\mathrm{Neg}(X)$ is the set of all $q_X$-exceptional prime divisors on $X$, i.e.
\[
\mathrm{Neg}(X):=\{D' \; | \; D' \mathrm{\; is \;a \;prime\; divisor\; and\; } q_X(D')<0\}.
\]
Indeed, let $D'$ be a prime divisor such that $q_X(D,D')=0$. Since $D$ is big, then $D \equiv_{\mathrm{num}} A+N$, where $A$ is an ample $\mathbb{R}$-divisor and $N$ is an effective $\mathbb{R}$-divisor (see Proposition 2.2.22 of \cite{Laz}). Thus we obtain
\begin{equation}\label{eq1}
q_X(D,D')=q_X(A+N,D')=q_X(A,N)+q_X(N,D')=0.
\end{equation}
Now, the class of $A$ in $N^1(X)_{\mathbb{R}} \subset H^{1,1}(X,\mathbb{R})$ is K\"ahler, thus $q_X(A,N)>0$.  The last equality in (\ref{eq1}) implies $q_X(N,D')<0$, and since $q_X$ is an intersection product, $D'$ must be one of the irreducible components of $N$, and $q_X(D')<0$.
\end{rmk}
\begin{lem}\label{lem1}
Let $P$ be  a big and $q_X$-nef $\mathbb{R}$-divisor on $X$. Then there is an open neighbourhood $U\subset \mathrm{Big}(X)$ of $[P]$ in $N^1(X)_{\mathbb{R}}$ such that for all classes $[D] \in U$ we have
\[
\mathrm{Null}_{q_X}(D) \subseteq \mathrm{Null}_{q_X}(P).
\]
\begin{proof}
Since the big cone is open, we can choose rational big classes $[D_1],\dots,[D_r]$ in $N^1(X)_{\mathbb{R}}$ such that  $P$ belongs to the interior of the closed cone  $K:=\sum_{i=1}^r \mathbb{R}^{\geq 0}[D_i]$. We can assume that $D_1,\dots,D_r$ are effective divisors, because a big class is effective. If $D \in K$, we can have $q_X(D,D')<0$ only for a finite number of prime divisors $D'$, because $q_X$ is an intersection product. Indeed, if $q_X(D,D')<0$, then $D'$ must be an irreducible component of $D$ . It follows that up to replacing each  $D_i$ with $\nu D_i$, for $\nu>0$ rational and small enough, we can assume that 
\begin{equation}\label{eq2}
q_X(P+D_i,D')>0
\end{equation}
for all the prime divisors $D'$ satisfying $q_X(P,D')>0$.
Now, consider an $\mathbb{R}$-divisor of the form $\sum_{i=1}^r\alpha_i(P+D_i)$, where $(\alpha_1,\dots,\alpha_r) \in \left(\mathbb{R}^{\geq 0}\right)^ r \setminus \left\{\vec{0}\right\}$. We note that if $D'$ is a prime divisor such that
\[
q_X \left(\sum_{i=1}^r\alpha_i(P+D_i),D'\right)=0,
\]
then $D' \in \mathrm{Null}_{q_X}(P)$, because otherwise, using that $P$ is $q_X$-nef and (\ref{eq2}), we would have $q_X(P+D_i,D')>0$ for each $i=1,\dots,r$. Then it is clear that
\[
\mathrm{Null}_{q_X}\left(\sum_{i=1}^r\alpha_i(P+D_i)\right) \subset \mathrm{Null}_{q_X}(P)
\]
for each $r$-tuple $(\alpha_1,\dots,\alpha_r ) \in \left(\mathbb{R}^{\geq 0} \right)^r \setminus \left\{ \vec{0} \right\}$. Therefore the open cone 
\[
U:=\sum_{i=1}^r \mathbb{R}^{> 0}[P+D_i]
\]
is an open neighbourhood of $P$ yielding the conclusion of the Lemma.
\end{proof}
\end{lem}

\begin{cor}\label{localpolyhed}
For every class $\alpha \in \mathrm{Big}(X) \cap \mathrm{Nef}_{q_X}(X)$ there exists an open neighbourhood $U=U(\alpha)$ in $N^1(X)_{\mathbb{R}}$ and prime divisors $D_1,\dots,D_n \in \mathrm{Neg}(X)$ such that
\[
U \cap\left( \mathrm{Big}(X) \cap \mathrm{Nef}_{q_X}(X)\right)= U \cap \left(D_1^{\geq 0}\cap \cdots \cap D_n^{\geq 0}\right).
\]
Moreover, $\overline{\mathrm{Nef}_{q_X}(X)\cap \mathrm{Big}(X)}=\mathrm{Nef}_{q_X}(X)$ is locally rational polyhedral at every big point.
\begin{proof}
Let $U$ be a neighborhood of $\alpha$ as in Lemma \ref{lem1}. Notice that by construction $\overline{U}$ is a rational cone. As every big class is effective and $q_X$ is an intersection product, we have the equality
\[
\mathrm{Big}(X) \cap \mathrm{Nef}_{q_X}(X)=\mathrm{Big}(X) \cap \left(\bigcap_{D \in \mathrm{Neg}(X)} D^{\geq 0}\right),
\]
and so
\begin{equation}\label{eq3}
\begin{split}
U \cap \left(\mathrm{Big}(X) \cap \mathrm{Nef}_{q_X}(X)\right)&= U \cap \mathrm{Big}(X) \cap \left(\bigcap_{D \in \mathrm{Neg}(X)} D^{\geq 0}\right)\\ &= U \cap  \left(\bigcap_{D \in \mathrm{Neg}(X)} D^{\geq 0}\right),
\end{split}
\end{equation}
because $U \subset \mathrm{Big}(X)$ by construction (see Lemma \ref{lem1}). We note that $\alpha$ belongs to the intersection in (\ref{eq3}). Moreover, if $D \in \mathrm{Neg}(X)$, either $U \subset D^{\geq 0}$ or $U \cap D^{\perp} \neq 0$. Indeed, if  $U \not \subset D^{\geq 0}$, we can find a class $[D'] \in U$ such that $q_X(D',D)<0$ and since $U$ is a convex cone, the segment joining $D'$ and $\alpha$ is contained in $U$ and must intersect $D^{\perp}$, hence $U \cap D^{\perp} \neq 0$.
Now, if  $U \subset D^{\geq 0}$ (for some $D \in \mathrm{Neg}(X)$), we can omit $D^{\geq 0}$ from (\ref{eq3}). Otherwise, we notice that we can have  $U \cap D^{\perp} \neq 0$ only for finitely many prime divisors in $\mathrm{Neg}(X)$. Indeed, if $[D'] \in U$, then $\mathrm{Null}_{q_X}(D')\subset \mathrm{Null}_{q_X}(\alpha)$ by construction, and $\mathrm{Null}_{q_X}(\alpha)$ is finite by Remark \ref{nulllocusfinite}. From the above argument, it follows that 
\[
U \cap \left(\mathrm{Big}(X) \cap \mathrm{Nef}_{q_X}(X)\right)= U \cap \left(\bigcap_{D \in \mathrm{Null}_{q_X}(\alpha)} D^{\geq 0} \right),
\]
hence $U$ is the open neighbourhood of $\alpha$ we were looking for, and $\mathrm{Nef}_{q_X}(X)$ is clearly locally rational polyhedral at $\alpha$ by definition.
\end{proof}
\end{cor}
 The following easy remarks will be useful in the rest of the paper.
\begin{rmk}
Let $D$ be an effective $\mathbb{R}$-divisor on $X$ and set
\[
\mathrm{Neg}_{q_X}(D)= \{N_1, \dots, N_k\}.
\]
We note that the classes $[N_1],\dots,[N_k]$ are linearly independent in $N^1(X)_{\mathbb{R}}$, because the Gram matrix $(q_X(N_i,N_j))_{i,j}$ is negative-definite.
\end{rmk}

\begin{rmk}\label{rmk1}
Let $D,D'$ be two distinct $q_X$-exceptional prime divisors on $X$. Then $[D] \neq [D']$ in $N^1(X)_{\mathbb{R}}$. Indeed, if $q_X(D,D') \geq 0$ and $[D]=[D']$, then $0>q_X(D)=q_X(D,D')\geq 0$, and this is clearly absurd. 
\end{rmk}

\begin{defn}[\bf Boucksom-Zariski chambers]
Let $P$ be a big and $q_X$-nef $\mathbb{R}$-divisor on $X$. The \textit{Boucksom-Zariski chamber} of $P$ is defined as
\[
\Sigma_P:=\left\{D \in \mathrm{Big}(X) \; | \; \mathrm{Neg}_{q_X}(D)=\mathrm{Null}_{q_X}(P)\right\}.
\]
\end{defn}
We note that this definition makes sense by item 2 in Remark \ref{rmk5}. Moreover, in general, these chambers are neither closed nor open.

\begin{ex}\label{exchamber}
Let $M$ be an $\mathbb{R}$-divisor whose class lies in $\mathrm{int}(\mathrm{Mov}(X))$. We observe that $M$ is also big. Indeed
\begin{equation}\label{equation10}
\overline{\mathrm{Mov}(X)}=\mathrm{Nef}_{q_X}(X)\subset \overline{\mathrm{Eff}(X)},
\end{equation}
where the equality in (\ref{equation10}) is justified by Remarks \ref{movableconeisqnefcone} and \ref{qnefrmk1}. It follows that
\[
\mathrm{int}\left(\overline{\mathrm{Mov}(X)}\right)=\mathrm{int}\left(\mathrm{Mov}(X)\right)\subset \mathrm{Big}(X),
\]
which implies that $M$ is also big, hence it makes sense to consider the chamber associated with $M$. Then $\Sigma_M=\mathrm{Big}(X)\cap \mathrm{Nef}_{q_X}(X)$ by definition.
\end{ex}

\begin{rmk}
If $P$ is big and $q_X$-nef, the chamber $\Sigma_P$ is a convex cone. Indeed, if $\alpha,\beta \in \Sigma_P$, 
\[
\mathrm{Null}_{q_X}(P)=\mathrm{Neg}_{q_X}(\alpha)=\mathrm{Neg}_{q_X}(\beta)=\mathrm{Neg}_{q_X}(\alpha+\beta),
\] 
hence $\alpha+\beta \in \Sigma_P$.
\end{rmk}

Recall that a face of a convex cone $K$ is a subcone $F$ such that if $v,v' \in K$, and $v+v' \in F$, then $v,v' \in F$. Let $P$ be a big and $q_X$-nef $\mathbb{R}$-divisor and define 
\[
\mathrm{Face}(P):=\mathrm{Nef}_{q_X}(X) \cap \mathrm{Null}_{q_X}(P)^{\perp},
\]
where 
\[
\mathrm{Null}_{q_X}(P)^{\perp}:=\bigcap_{\substack{D' \in \mathrm{Null}_{q_X}(P)}}(D')^{\perp}.
\]
We observe that $\mathrm{Face}(P)$ is a face of the $q_X$-nef cone. Indeed, suppose $\alpha+\beta \in \mathrm{Face}(P)$, where $\alpha,\beta \in \mathrm{Nef}_{q_X}(X)$. If $D' \in \mathrm{Null}_{q_X}(P)$, then 
\[
0=q_X(\alpha+\beta,D')=q_X(\alpha,D')+q_X(\beta,D').
\]
 Thus, since $\alpha$ and $\beta$ are $q_X$-nef, we must have $q_X(\alpha,D')=q_X(\beta,D')=0$, and so $\alpha, \beta \in \mathrm{Face}(P)$.

\begin{defn}\label{qexcblock}
A subset $S \subset \mathrm{Neg}(X)$ is a \textit{$q_X$-exceptional block} if $S=\{\emptyset\}$, or $S=\{D_1,\dots,D_k\}$ and the Gram matrix $\left(q_X(D_i,D_j)\right)_{i,j}$ is negative-definite.
\end{defn}
 
 \begin{lem}\label{lemexcepblock}
 Let $S$ be a $q_X$-exceptional block. Then there exists a $q_X$-nef divisor $P$ satisfying $\mathrm{Null}_{q_X}(P)=S$.
 \begin{proof}
If $S=\{\emptyset\}$, any divisor in $\mathrm{int}\left(\mathrm{Mov}(X)\right)$ will satisfy what is wanted (and we have such a divisor because $X$ is projective). If $S=\{D_1,\dots,D_k\}$, pick a divisor $M\in \mathrm{int}\left(\mathrm{Mov}(X)\right)$. We claim that $P$ can be constructed explicitly, and of the form $P=M+\sum_{i=1}^k\lambda_i D_i$, where the $\lambda_i$ are positive rational numbers. To prove the claim it is sufficient to check that:
\begin{enumerate}
\item the linear system $\mathcal{S}$ of equations
\[
q_X(M,D_j)+\sum_{i=1}^k q_X( D_j,D_i)\lambda_i=0, \; \; j=1,\dots,k 
\] 
admits a (unique) solution $(\lambda_1,\dots,\lambda_k)\in \left(\mathbb{Q}^{>0}\right)^k$;
\item with such a solution $\mathrm{Null}_{q_X}(P)=S$ holds.
\end{enumerate}
Let us prove (1). For this purpose, we will need Lemma 4.1 of \cite{Bau}. We know that $G$ is negative-definite and so invertible. Then $\mathcal{S}$ has a unique solution. Now, $\mathcal{S}$ can be rewritten as
\[
G^{-1} \cdot (q_X(M,D_1),\dots,q_X(M,D_k))^{T}=(-\lambda_1,\dots,-\lambda_k)^T.
\]
The entries of $G^{-1}$ are rational, and by Lemma 4.1 of \cite{Bau} are also non-positive. Moreover, each $q_X(M,D_j)$ is positive, because any element in $\mathrm{int}\left(\mathrm{Mov}(X)\right)$ intersects positively (with respect to $q_X$) any $q_X$-exceptional prime divisor . From all the above it follows that the $\lambda_i$ are positive and rational, and we are done. We now have to prove (2). In particular, as $S \subseteq \mathrm{Null}_{q_X}(P)$ by construction, we only have to verify that $\mathrm{Null}_{q_X}(P) \subseteq S$. But this is immediate, because $P=M+\sum_{i=1}^k \lambda_iD_i$, and $q_X(P,D')=0$ (where $D' \in \mathrm{Null}_{q_X}(P)$) clearly implies $D' \in S$.
  \end{proof}
 \end{lem}
\begin{lem}\label{chambersequaliff}
Let $P$ and $P'$ be big and $q_X$-nef $\mathbb{R}$-divisors on $X$. The following hold:
\begin{enumerate}
\item $\Sigma_P=\Sigma_{P'}$ if and only if $\mathrm{Face}(P)=\mathrm{Face}(P')$.
\item $\Sigma_P \cap \Sigma_{P'}= \emptyset$, if $\mathrm{Face}(P) \neq \mathrm{Face}(P')$.
\item $\mathrm{Big}(X)$ is the disjoint union of the Boucksom-Zariski chambers.
\end{enumerate}
\begin{proof}

\begin{enumerate}
\item We first show $\Leftarrow$. By definition of $\Sigma_P$ and $\Sigma_{P'}$, it is sufficient to prove that $\mathrm{Null}_{q_X}(P)=\mathrm{Null}_{q_X}(P')$. If $P \in \mathrm{Face}(P)=\mathrm{Face}(P')$, then $q_X(P,D)=0$, for every prime divisor $D$ belonging to $\mathrm{Null}_{q_X}(P')$. Thus $\mathrm{Null}_{q_X}(P') \subset \mathrm{Null}_{q_X}(P)$, and by symmetry also the other inclusion holds. The implication $\Rightarrow$ is straightforward.
\item The contrapositive of $\Rightarrow$ in item 1 tells us that if $\mathrm{Face}(P) \neq \mathrm{Face}(P')$, then $\Sigma_P \neq \Sigma_{P'}$, and this clearly implies $\Sigma_P \cap \Sigma_{P'}= \emptyset$.
\item We have the inclusion
\[
\bigcup_{\substack{P \in \mathrm{Big}(X) \cap \mathrm{Nef}_{q_X}(X)}} \Sigma_P \subseteq \mathrm{Big}(X).
\]
The inclusion $\supseteq$ follows directly applying Lemma \ref{lemexcepblock} to $\mathrm{Neg}_{q_X}(D)$, where $D$ is a big $\mathbb{R}$-divisor.
\end{enumerate}
\end{proof}
\end{lem}

 \begin{cor}\label{otobzc}
 There is a one-to-one correspondence between the Boucksom-Zariski chambers and the $q_X$-exceptional blocks.
 \begin{proof}
 This is clear from the definition of Boucksom-Zariski chamber, from Lemma \ref{lemexcepblock} and item (3) of Lemma \ref{chambersequaliff}. 
 \end{proof}
  \end{cor}

\begin{prop}\label{propborder}
A big class $\alpha$ is on the boundary of some chamber $\Sigma_P$, with $P$ a big and $q_X$-nef $\mathbb{R}$-divisor, if and only if $\mathrm{Neg}_{q_X}(\alpha) \subsetneq \mathrm{Null}_{q_X}(P(\alpha))$.
\begin{proof}
As usual we can assume $\alpha=[D]$, with $D$ a big and effective $\mathbb{R}$-divisor. We first show $\Rightarrow$. Let $D \in \partial \Sigma_P$ for some big and $q_X$-nef divisor $P$, and $N(D)=\sum_{i=1}^ka_i N_i$. Let $\norm{\cdot} $ be any norm on $N^1(X)_{\mathbb{R}}$. As $D$ is on the boundary of $\Sigma_P$ (of course this does not imply $D \in \Sigma_P$), any small "ball" $B_{\epsilon}(D) \subset \mathrm{Big}(X)$ of center $D$ and radius $\epsilon$ will contain an element from another chamber. Thus, for every $\epsilon >0$ small enough, we can find a class $\beta \in N^1(X)_{\mathbb{R}}$ of norm $\norm{\beta}< \epsilon$, such that $\mathrm{Neg}_{q_X}(D+\beta) \neq \mathrm{Neg}_{q_X}(D)$. As $q_X$ is non-degenerate on $N^1(X)_{\mathbb{R}}$, we can  decompose the N\'{e}ron-Severi space as
\[
N^1(X)_{\mathbb{R}}= \left<P(D),N_1,\dots,N_K\right> \oplus \left<P(D),N_1,\dots,N_k\right>^{\perp},
\]
and so we can write $\beta= \beta' \oplus \beta''$, where 
\[
\beta' \in \left<P(D),N_1,\dots,N_K\right> \; \mathrm{ and } \; \beta'' \in \left<P(D),N_1,\dots,N_K \right>^{\perp}
\] 
(note that $P(D),N_1,\dots,N_k$ are linearly independent in $N^1(X)_{\mathbb{R}}$). If necessary, we can choose $\epsilon$ even smaller, in order to have 
\[
\mathrm{Neg}_{q_X}(D+ \beta')=\mathrm{Neg}_{q_X}(D).
\] 
It follows that we can assume $\beta \in \left< P(D),N_1,\dots,N_k\right>^{\perp}$. Now, as by Lemma \ref{positivepartisbig} $P(D)$ is big and $q_X$-nef, we can take an open neighborhood $U=U(P(D))$ like the one in Lemma \ref{lem1}. As $U$ is open,  choosing $\epsilon$ even smaller (if necessary), we may assume that $P(D)+ \beta \in U$. Notice that, by construction, $\mathrm{Neg}_{q_X}(D+ \beta) \neq \mathrm{Neg}_{q_X}(D)$. Thus $P(D)+ \beta$ cannot be $q_X$-nef, as otherwise the divisorial Zariski decomposition of $D+ \beta$ would be $(P(D)+\beta)+N(D)$, because $N(D)$ is $q_X$-orthogonal to $\beta$, and so we would have $\mathrm{Neg}_{q_X}(D+ \beta) = \mathrm{Neg}_{q_X}(D)$, which is absurd. It follows that we can find a prime divisor $D' \in \mathrm{Neg}(X) \setminus \{N_1, \dots, N_k\}$ such that $q_X(P(D)+\beta,D')<0$. But $q_X(P(D),D') \geq 0$, because $P(D)$ is $q_X$-nef. Hence there exists $t_0 \in [0,1[$ satisfying 
\[
q_X(P(D)+t_0\beta,D')=q_X(P(D),D')+t_0q_X(\beta,D')=0.
\]
By the choice of $U$ we know that $\mathrm{Null}_{q_X}(P(D)+t_0\beta)\subset \mathrm{Null}_{q_X}(P(D))$, and so $D' \in \mathrm{Null}_{q_X}(P(D))$. Thus we obtained $\mathrm{Neg}_{q_X}(D) \subsetneq \mathrm{Null}_{q_X}(P(D))$.

The implication $\Leftarrow$ is easier. Without loss of generality we can assume $\mathrm{Null}_{q_X}(P(D)) \setminus \mathrm{Neg}_{q_X}(D)=\{D'\}$. Now, consider $D+ \epsilon D'$, where $\epsilon >0$. It is clear that the divisorial Zariski decomposition of $D+ \epsilon D'$ is $P(D)+(N(D)+\epsilon D')$, and so $D+\epsilon D'$ lies in a different chamber than the one to which $D$ belongs (in particular, this chamber is exactly $\Sigma_{P(D)}$), for every $\epsilon >0$. But $\mathrm{lim}_{\epsilon \to 0} (D+\epsilon D')=D$, thus $D$ is a limit point of a sequence contained in the chamber $\Sigma_{P(D)}$, hence it must lie in the boundary of $\Sigma_{P(D)}$.
\end{proof}
\end{prop}

\begin{cor}\label{interiorchamber}
For every big and $q_X$-nef $\mathbb{R}$-divisor, the interior of a Boucksom-Zariski chamber $\Sigma_P$ is
\[
\mathrm{int}\left(\Sigma_P\right)=\left\{D \in \mathrm{Big}(X) \; | \; \mathrm{Neg}_{q_X}(D)=\mathrm{Null}_{q_X}(P)=\mathrm{Null}_{q_X}(P(D))\right\}.
\]
\begin{proof}
 The inclusion $\subseteq$ follows from the definition of $\Sigma_P$ and Proposition \ref{propborder}. As for $\supseteq$, if $D$ is such that $\mathrm{Neg}_{q_X}(D)=\mathrm{Null}_{q_X}(P)=\mathrm{Null}_{q_X}(P(D))$, clearly $D \in \Sigma_P$. Moreover $D \in \mathrm{int}(\Sigma_P)$, again thanks to Proposition $\ref{propborder}$.
\end{proof}

\end{cor}

\begin{rmk}\label{rmkbackzariskidecomp}
From Corollary \ref{stronghodgeindex} and Lemma \ref{bigqnefispositive} it follows that if $P$ is a big and $q_X$-nef integral divisor, and if $\{D_1,\dots,D_k\} \subset \mathrm{Null}_{q_X}(P)$, then each (non-zero) effective linear combination of the $D_i$ is $q_X$-exceptional. In particular, if we choose nonnegative real numbers $a_1,\dots,a_k$, then $D':=P+(\sum_{i=1}^k a_iD_i)$ is exactly the divisorial Zariski decomposition of $D'$, i.e. $P(D')=P$ and $N(D')=\sum_{i=1}^ka_iD_i$.
\end{rmk}

\begin{defn}
Let $P$ be a big and $q_X$-nef divisor. We define the \textit{relative interior} of $\mathrm{Face}(P)$ as the interior of $\mathrm{Face}(P)$ in the topology of $\mathrm{Null}_{q_X}(P)^{\perp}$. We will denote it by $\mathrm{rel.int.Face}(P)$.
\end{defn}

\begin{ex}
Let $M$ be an $\mathbb{R}$-divisor belonging to $\mathrm{int}(\mathrm{Mov}(X))$. Then $\mathrm{rel.int.Face}(M)=\mathrm{int}(\mathrm{Nef}_{q_X}(X))=\mathrm{int}(\mathrm{Mov}(X))$. Indeed $\mathrm{Null}_{q_X}(M)^{\perp}=N^1(X)_{\mathbb{R}}$, and so $\mathrm{rel.int.Face}(M)$ is the interior of the $q_X$-nef cone in the topology of $N^1(X)_{\mathbb{R}}$.
\end{ex}

\begin{rmk}\label{rmkrelint}
With the notation of the above definition, we note the following facts.
\begin{enumerate}
 \item First of all we have $P \in \mathrm{rel.int.Face}(P)$. Indeed, set $\{D_1,\dots,D_k\}=\mathrm{Null}_{q_X}(P)$. Thanks to Corollary \ref{localpolyhed}, we can find an open neighbourhood $U=U(P)\subset \mathrm{Big}(X)$ such that 
\[
U \cap \mathrm{Nef}_{q_X}(X)= U \cap (D_1^{\geq0}\cap \cdots \cap D_k^{\geq 0}).
\] 
It follows that 
\[
U \cap \mathrm{Null}_{q_X}(P)^{\perp} \subset  U \cap (D_1^{\geq0}\cap \cdots \cap D_k^{\geq 0}),
\]
thus we have found an open neighbourhood of $P$ (which is clearly non-empty as $P$ belongs to it) in $\mathrm{Null}_{q_X}(P)^{\perp}$ contained in $\mathrm{Nef}_{q_X}(X)$, i.e. $P \in \mathrm{rel.int.Face}(P)$.

\item It is important to observe that since $Q \in \mathrm{rel.int.Face}(P)$, then we must have $\mathrm{Null}_{q_X}(Q)=\mathrm{Null}_{q_X}(P)$. Indeed $\mathrm{Null}_{q_X}(P) \subseteq \mathrm{Null}_{q_X}(Q)$ and also $\supseteq$ must hold, as otherwise, for $D' \in \mathrm{Null}_{q_X}(Q)\setminus \mathrm{Null}_{q_X}(P)$, we would have $q_X(Q -\epsilon P, D')<0$ for every $\epsilon>0$, thus $Q -\epsilon P$ would not be $q_X$-nef and so $Q$ could not lie in $\mathrm{rel.int.Face}(P)$ (because otherwise $Q -\epsilon P$ would have been $q_X$-nef for every $\epsilon$ small enough).
\end{enumerate}
\end{rmk}

\begin{lem}\label{lemrelintface}
Let $P$ a big and $q_X$-nef $\mathbb{R}$-divisor on $X$, and $Q \in \mathrm{Face}(P)$. Then $Q$ is big if $Q \in \mathrm{rel.int.Face}(P)$. 
\begin{proof}
We can assume that $P$ is integral by the proof of Lemma \ref{chambersequaliff}. Since $Q \in \mathrm{rel.int.Face}(P)$, we can find $0<\epsilon_0 \ll 1$ such that $Q-\epsilon_0 P$ is $q_X$-nef and so pseudo-effective. Indeed, by Remarks \ref{movableconeisqnefcone} and \ref{qnefrmk1}, $\overline{\mathrm{Mov}(X)}=\mathrm{Nef}_{q_X}(X)$, and $\mathrm{Mov}(X) \subseteq \mathrm{Eff}(X)$ by definition, hence $\mathrm{Nef}_{q_X}(X) = \overline{\mathrm{Mov}(X)} \subseteq \overline{\mathrm{Eff}(X)}$.  Set $\{D_1,\dots,D_k\}=\mathrm{Null}_{q_X}(P)=\mathrm{Null}_{q_X}(Q)$. Now,  $P$ is big, so $P+\frac{1}{\epsilon_0}(a_1D_1+\cdots+a_kD_k)$ is big for every nonnegative $a_1,\dots,a_k \in \mathbb{R}$. But 
\[
(Q-\epsilon_0 P)+(\epsilon_0 P + a_1D_1+\cdots +a_kD_k)=Q+a_1D_1+\cdots+a_kD_k=:Q_{a_1\dots a_k}
\]
is big, as it is the sum of a pseudo-effective divisor and a big divisor. Moreover, the positive part of the divisorial Zariski decomposition of $Q_{a_1 \dots a_k}$ is exactly $Q$ and its negative part is exactly $\sum_{i=1}^ka_iD_i$, by Remark \ref{rmkbackzariskidecomp}. Thus $Q$ must be big by Lemma \ref{positivepartisbig}.
\end{proof}
\end{lem}
We now describe the big part of the closure of a given chamber in terms of $q_X$-null loci and $q_X$-negative loci.
\begin{prop}\label{analyticcharact}
Let $P$ be a big and $q_X$-nef $\mathbb{R}$-divisor. Then
\begin{equation}\label{equationchamber}
\mathrm{Big}(X) \cap \overline{\Sigma_{P}}= \left\{D \in \mathrm{Big}(X) \; | \; \mathrm{Neg}_{q_X}(D) \subset \mathrm{Null}_{q_X}(P) \subset \mathrm{Null}_{q_X}(P(D))\right\}
\end{equation}
\begin{proof}
Again, by Lemma \ref{chambersequaliff}, we can assume that $P$ is integral. We first show $\supseteq$. Pick $D=P(D)+N(D)$ like in the right-hand side of (\ref{equationchamber}), where as usual $P(D)+N(D)$ is the divisorial Zariski decomposition of $D$.
As $\mathrm{Null}_{q_X}(P) \subset \mathrm{Null}_{q_X}(P(D))$, then $P(D) \in \mathrm{Face}(P)$. Indeed we have $\mathrm{Null}_{q_X}(P(D))^{\perp} \subset \mathrm{Null}_{q_X}(P)^{\perp}$ and so $P(D) \in \mathrm{Face}(P(D)) \subset \mathrm{Face}(P)$. By Lemma \ref{lemrelintface} we can approximate $P(D)$ with a sequence $\{Q_n\}_n$ of big divisors contained in $\mathrm{rel.int.Face}(P) \neq \emptyset$.
Moreover, by Remark \ref{rmkrelint}, we also have $\mathrm{Null}_{q_X}(Q_n)=\mathrm{Null}_{q_X}(P)$. On the other hand, if $\{D_1,\dots,D_j\}=\mathrm{Null}_{q_X}(P) \setminus \mathrm{Neg}_{q_X}(D)$,  then , by Remark \ref{rmkbackzariskidecomp},
\[
\left\{N_n:=N(D)+\frac{1}{n}\sum_{i=1}^jD_i \right\}_n 
\]
 is a sequence of $q_X$-exceptional divisors converging to $N(D)$.
If we let $D_n:=Q_n+N_n$, we obtain a sequence $\{D_n\}_n$ converging to $D$ and $Q_n=P(D_n)$ (resp. $N(D_n)=N_n$) is the positive (resp. negative) part of the divisorial Zariski decomposition of $D_n$. So, by construction, 
\[
\mathrm{Neg}_{q_X}(D_n)=\mathrm{Null}_{q_X}(P)=\mathrm{Null}_{q_X}(Q_n)=\mathrm{Null}_{q_X}(P(D_n)).
\]
 By Corollary \ref{interiorchamber}, each $D_n$ belongs to $\mathrm{int}(\Sigma_P)$, and so $D \in \mathrm{Big}(X)\cap \overline{\Sigma_P}$.
 
 Now we prove $\subseteq$. Let $\{D_n\}_n$ be a sequence of divisors in $\mathrm{int}(\Sigma_P)$ converging to $D$. If $D_n=P(D_n)+N(D_n)$ is the divisorial Zariski decomposition of each term in the sequence, then
 \[
 \mathrm{Neg}_{q_X}(D_n)=\mathrm{Null}_{q_X}(P)=\mathrm{Null}_{q_X}(P(D_n)),
 \]
 which in turn implies
 \[
 P(D_n) \in \mathrm{Face}(P) \mathrm{ \;and\; } N(D_n) \in \left<\mathrm{Null}_{q_X}(P)\right>.
 \]
 Since $\mathrm{Null}_{q_X}(P)^{\perp}$ and $\left<\mathrm{Null}_{q_X}(P)\right>$ are closed and orthogonal, it follows that $\{P(D_n)\}_n$ converges to a $q_X$-nef $\mathbb{R}$-divisor belonging to $\mathrm{Null}_{q_X}(P)^{\perp}$, and $\{N(D_n)\}_n$ converges to an element in $\left<\mathrm{Null}_{q_X}(P)\right>$ . Now, if we set $\mathrm{Null}_{q_X}(P)=\{D_1,\dots, D_k\}$, we can write $N(D_n)=\sum_{i=1}^ka^{(n)}_i D_i$, and each $a^{(n)}_i$ is positive. As the $D_j$ are linearly independent in $N^1(X)_{\mathbb{R}}$, it follows that $\{N(D_n)\}_n$ converges to an element of the form $\sum_{i=1}^ka_iD_i$, where some $a_i$ can be $0$. From all the above it is clear that 
 \[
 \{P(D_n)\}_n\to P(D) \in \mathrm{Null}_{q_X}(P)^{\perp} \mathrm{ \;and \;} \{N(D_n)\}_n \to N(D).
 \] 
 The first condition gives $\mathrm{Null}_{q_X}(P) \subset \mathrm{Null}_{q_X}(P(D))$, and the fact that $N(D)=\sum_{i=1}^ka_iD_i$ (where some $a_i$ can be $0$) implies that $\mathrm{Neg}_{q_X}(D) \subset \mathrm{Null}_{q_X}(P)$. 
\end{proof}
\end{prop}
We are ready to describe geometrically the big part of the closure of any chamber. In what follows $V^{\geq 0}(M)$ will denote the subcone of $N^1(X)_{\mathbb{R}}$ generated by a subset $M \subset N^1(X)_{\mathbb{R}}$ and $V^{>0}(M)$ its interior.
\begin{prop}\label{charactchambers}
Let $P$ be a big and $q_X$-nef $\mathbb{R}$-divisor. We have
\[
\mathrm{Big}(X) \cap \overline{\Sigma_P}=\left(\mathrm{Big}(X) \cap \mathrm{Face}(P)\right)+ V^{\geq 0}\left(\mathrm{Null}_{q_X}(P)\right).
\]
\begin{proof}
As usual, we can assume that $P$ is integral. We first prove $\subseteq$. Let $D \in \mathrm{Big}(X) \cap \overline{\Sigma_P}$.  By Proposition \ref{analyticcharact} we know that this holds if and only if
\[
\mathrm{Neg}_{q_X}(D)\subset \mathrm{Null}_{q_X}(P) \subset \mathrm{Null}_{q_X}(P(D)).
\]
Then $P(D) \in \mathrm{Null}_{q_X}(P(D))^{\perp} \subset \mathrm{Null}_{q_X}(P)^{\perp}$. We also know that $P(D)$ is big by Lemma \ref{positivepartisbig} and $q_X$-nef by construction, thus $P(D) \in \mathrm{Face}(P)\cap \mathrm{Big}(X)$. Now, $\mathrm{Neg}_{q_X}(D) \subset \mathrm{Null}_{q_X}(P)$, thus $N(D) \in V^{\geq 0}\left(\mathrm{Null}_{q_X}(P)\right)$, hence $D \in \left(\mathrm{Big}(X) \cap \mathrm{Face}(P)\right)+ V^{\geq 0}\left(\mathrm{Null}_{q_X}(P)\right)$.

Now we prove $\supseteq$. Suppose that 
\[
D \in \left(\mathrm{Big}(X) \cap \mathrm{Face}(P)\right) + V^{\geq 0}\left(\mathrm{Null}_{q_X}(P)\right).
\]
Write 
\[
D=Q+M,
\]
where $Q \in \mathrm{Big}(X) \cap \mathrm{Face}(P)$, and $M \in V^{\geq 0}\left(\mathrm{Null}_{q_X}(P)\right)$. Clearly $Q$ and $M$ are $q_X$-orthogonal, by definition of $\mathrm{Face}(P)$. Moreover, $M$ is $q_X$-exceptional, by Corollary \ref{stronghodge}, and $Q$ is $q_X$-nef. It follows that $Q=P(D)$ and $M=N(D)$, by the uniqueness of the divisorial Zariski decomposition. Of course we have $\mathrm{Neg}_{q_X}(D)=\mathrm{Neg}_{q_X}(M)\subset \mathrm{Null}_{q_X}(P)$, and to conclude, as $Q=P(D) \in \mathrm{Null}_{q_X}(P)^{\perp}$, we also have $\mathrm{Null}_{q_X}(P) \subset \mathrm{Null}_{q_X}(P(D))$, and so we are done by Proposition \ref{analyticcharact}.
\end{proof}
\end{prop}

\begin{cor}\label{corlocpolyhed}
Let $P$ be a big and $q_X$-nef $\mathbb{R}$-divisor. Then $\overline{\Sigma_P}$ is locally rational polyhedral at every big point.
\end{cor}
\begin{proof}
We can assume that $P$ is integral, because, according to the proof of Lemma \ref{lemexcepblock}, any Boucksom-Zariski chamber is associated with a big and $q_X$-nef integral divisor. Let $\alpha \in \mathrm{Big(X)}\cap \overline{\Sigma}_P$ and write $\alpha=P(\alpha)+N(\alpha)$ for its divisorial Zariski decomposition. Moreover, set $\mathrm{Null}_{q_X}(P)=\{D_1,\dots,D_k\}$. By Proposition \ref{charactchambers} we have 
\[
\mathrm{Big}(X) \cap \overline{\Sigma_P}=\mathrm{Face}(P) \cap \mathrm{Big}(X)+ V^{\geq 0}\left(\mathrm{Null}_{q_X}(P)\right).
\]
The part $V^{\geq 0}\left(\mathrm{Null}_{q_X}(P)\right)$ is clearly rational. We know that $P(\alpha) \in \mathrm{Face}(P) \cap \mathrm{Big}(X)$, by Lemma \ref{positivepartisbig}. Let $U=U(P(\alpha))$ be a rational neighborhood like the one in Lemma \ref{lem1}. Then, by Corollary \ref{localpolyhed}, we obtain
\begin{equation}\label{equationlocpolyhed}
U \cap \mathrm{Face}(P) \cap \mathrm{Big}(X)=U \cap \mathrm{Nef}_{q_X}(P) \cap \mathrm{Null}_{q_X}(P)^{\perp}= U \cap \left(D_1^{\perp}\cap \dots \cap D_k^{\perp}\right).
\end{equation}
By (\ref{equationlocpolyhed}), the neighborhood $\left(U+V^{\geq 0}\left(\mathrm{Null}_{q_X}(P)\right)\right) \cap \overline{\Sigma_{P}}$ of $\alpha$ is cut out by finitely many rational closed half-spaces, hence we are done.
\end{proof}
We now give a geometric description of the interior of any chamber. 

\begin{prop}
Let $P$ be a big and $q_X$-nef $\mathbb{R}$-divisor. The interior of the chamber $\Sigma_P$ is equal to
\[
\mathrm{rel.int.Face}(P)+V^{>0}\left(\mathrm{Null}_{q_X}(P)\right).
\]
\begin{proof}
We can assume that $P$ is integral, because, according to the proof of Lemma \ref{lemexcepblock}, any Boucksom-Zariski chamber is associated with a big and $q_X$-nef integral divisor. Let $D \in \mathrm{int}(\Sigma_P)$. By Corollary \ref{interiorchamber} we have
\[
\mathrm{Neg}_{q_X}(D)=\mathrm{Null}_{q_X}(P)=\mathrm{Null}_{q_X}(P(D)).
\]
Then $\mathrm{Face}(P)=\mathrm{Face}(P(D))$, hence, by item 1 of Remark \ref{rmkrelint}, $P(D) \in \mathrm{rel.int.Face}(P)$. Moreover $N(D) \in V^{>0}(\mathrm{Null}_{q_X}(P))$, because $\mathrm{Neg}_{q_X}(D)=\mathrm{Null}_{q_X}(P)$, thus
\[
D \in \mathrm{rel.int.Face}(P)+V^{>0}\left(\mathrm{Null}_{q_X}(P)\right).
\]
Now, assume that $D \in \mathrm{rel.int.Face}(P)+V^{>0}\left(\mathrm{Null}_{q_X}(P)\right)$. Write 
\[
D=Q+M,
\]
where $Q \in \mathrm{rel.int.Face}(P)$ and $M \in V^{>0}\left(\mathrm{Null}_{q_X}(P)\right)$. By Lemma \ref{lemrelintface} $Q$ is big, hence $D$ is big, as it is the sum of a big divisor and an effective divisor. By Remark \ref{rmkbackzariskidecomp} $M$ is $q_X$-exceptional. Moreover $Q$ is $q_X$-nef, and $q_X(Q,M)=0$. It follows from the uniqueness of the divisorial Zariski decomposition that $Q=P(D)$ and $M=N(D)$, i.e. $D=Q+M$ is the divisorial Zariski decomposition of $D$. Clearly we have $\mathrm{Neg}_{q_X}(D)=\mathrm{Null}_{q_X}(P)$, and by item 2 of Remark \ref{rmkrelint} we also have $\mathrm{Null}_{q_X}(P)=\mathrm{Null}_{q_X}(P(D))$. Thus we have
\[
\mathrm{Neg}_{q_X}(D)=\mathrm{Null}_{q_X}(P)=\mathrm{Null}_{q_X}(P(D)),
\]
and by Corollary \ref{interiorchamber} we are done.
\end{proof}
\end{prop}

Our next goal is to prove that the decomposition of $\mathrm{Big}(X)$ into Boucksom-Zariski chambers is locally finite, i.e. for each point in $\mathrm{Big}(X)$ there exists a neighborhood meeting only a finite number of Boucksom-Zariski chambers. The key to proving this result is the following.

\begin{rmk}\label{locallyfinitelem}
Let $D$ be a big $\mathbb{R}$-divisor and $A$ an ample $\mathbb{R}$-divisor. Then
\[
\mathrm{Neg}_{q_X}(D+ A) \subset \mathrm{Neg}_{q_X}(D).
\]
 This follows directly from the fact that $P(D+A)$ is the maximal $q_X$-nef subdivisor of $D+A$, and $P(D+A)\geq P(D)+A$.
\end{rmk}

\begin{prop}\label{localfinitenessprop}
The decomposition of the big cone of $X$ into Boucksom-Zariski chambers is locally finite.
\begin{proof}
Let $\mathrm{Amp}(X)$ be the ample cone of $X$. It is sufficient to note that every big $\mathbb{R}$-divisor $D$ has an open neighborhood of the form $D'+\mathrm{Amp}(X)$. Indeed $D\equiv_{\mathrm{num}}A+N$, where $A$ is ample and $N$ is effective.
We can write $A=\sum_{i}a_i A_i$, where the $a_i$ are positive, and the $A_i$ are integral and ample. If we choose $A'=\sum_{i=1}^kb_iA_i$, with $0<b_i<a_i$ for every $i$, then $D'=(A-A')+N$ is still big, and $[D] \in [D']+\mathrm{Amp}(X) \subset N^1(X)_{\mathbb{R}}$. By Remark \ref{locallyfinitelem}, every class $\alpha$ in this neighbourhood satisfy $\mathrm{Neg}_{q_X}(\alpha) \subset \mathrm{Neg}_{q_X}(D)$, and this clearly implies that $[D']+ \mathrm{Amp}(X)$ meets only a finite number of Boucksom-Zariski chambers.
\end{proof}
\end{prop}

We are now ready to prove item 1 of Theorem \ref{mainthm1}.
\bigskip
\begin{proof2}
By Proposition \ref{chambersequaliff}, $\mathrm{Big}(X)$ is the disjoint union of the Boucksom-Zariski chambers, and by definition in each chamber the support of the negative part of the divisorial Zariski decomposition of the divisors in constant. By Corollary \ref{localpolyhed} and Corollary \ref{corlocpolyhed}, every Boucksom-Zariski chamber is locally rational polyhedral, and by Proposition \ref{localfinitenessprop} the decomposition into Boucksom-Zariski chambers is locally finite.
\end{proof2}

The below result shows how the local finiteness of the given decomposition allows us to prove that the divisorial Zariski decomposition varies "continuously" in the big cone.

\begin{prop}\label{continuityboucksomZariski}
Let $\{[D_n]\}_{n}$ be a sequence of big classes converging in $N^1(X)_{\mathbb{R}}$ to a big class $[D]$. If $D_n=P(D_n)+N(D_n)$ is the divisorial Zariski decomposition of $D_n$, and if $D=P(D)+N(D)$ is the divisorial Zariski decomposition of $D$, then $\{[P(D_n)]\}_n$ (respectively $\{[N(D_n)]\}_n$) converges to $[P(D)]$ (respectively $[N(D)]$).
\begin{proof}
Assume first that $\{[D_n]\}_n \subset \Sigma_M$, for some $M$ big and $q_X$-nef. We note that $\{[P(D_n)]\}_n \subset \mathrm{Null}_{q_X}(M)^{\perp}$, hence $\{[P(D_n)]\}_n \to [P] \in \mathrm{Null}_{q_X}(M)^{\perp} \cap \mathrm{Nef}_{q_X}(X)$. Moreover $\{[N(D_n)]\}_n \subset \left<\mathrm{Null}_{q_X}(M)\right>$, hence $\{[N(D_n)]\}_n \to [N] \in \left<\mathrm{Null}_{q_X}(M)\right>$. Set $\mathrm{Null}_{q_X}(M)=\{N_1,\dots,N_k\}$. For each $n$ we can write $N(D_n)=\sum_{i=1}^ka_i^{(n)}N_i$, where $a_i^{(n)}>0$ for each $n$, and for each $i=\{1,\dots,k\}$.  Moreover the classes $[N_i]$ are linearly independent in $N^1(X)_{\mathbb{R}}$, because the matrix $\left(q_X(N_i,N_j)\right)_{i,j}$ is negative-definite. Hence each $\left\{a_i^{(n)}\right\}_n$ converges to a number $a_i\geq 0$, thus we can write  $N=\sum_{i=1}^k a_iN_i$. Putting it all together we obtained:
\begin{enumerate}
\item $q_X(P,N)=0$,
\item $\sum_{i=1}^k a_iN_i=N$ is $q_X$-exceptional,
\item $P$ is $q_X$-nef.
\end{enumerate} 
Then it is clear that $[P]=[P(D)]$ and $N=N(D)$. Assume now that $\{[D_n]\}_{n}$ is contained in more than one chamber. We know by Lemma \ref{locallyfinitelem} that $[D]$ has a neighborhood of the form $[D']+ \mathrm{Amp}(X)$, with $[D']$ a big class, and this neighborhood meets only a finite number of Boucksom-Zariski chambers, by Proposition \ref{localfinitenessprop}. Then one of these chambers must contain an infinite number of terms of $\{[D_n]\}_n$. Hence we are reduced to the case in which $\{[D_n]\}_n$ is contained in one chamber, and we are done.
\end{proof} 
\end{prop}

The rest of this section is devoted to proving item $2$ of Theorem \ref{mainthm1}. In particular, we show that the volume function on $\mathrm{Big}(X)$ is locally polynomial. Note that Bauer, K\"uronya and Szemberg in Subsection 3.3 of \cite{Bau} provided a smooth projective threefold whose associated volume function is not locally polynomial. The next two results are crucial for our purposes.

\begin{lem}[Proposition 3.20 in \cite{Bouck}]\label{volumeofdivisors}
Let $D$ be a big $\mathbb{R}$-divisor on $X$, and $P(D)+N(D)$ its divisorial Zariski decomposition. Then $\mathrm{vol}(D)=\mathrm{vol}(P(D))$.
\end{lem}

By Lemma \ref{volumeofdivisors} we know that $\mathrm{vol}(D)=\mathrm{vol}(P(D))$, moreover $P(D)$ is big and $q_X$-nef. For a big and nef $\mathbb{R}$-divisor we know that the volume is given by the top self-intersection (see equation (2.9) in \cite{Laz}), and a priori we cannot say the same for $P(D)$, as it is "only" $q_X$-nef, and of course we could have $\mathrm{Nef}(X) \subsetneq \mathrm{Nef}_{q_X}(X)$. The upshot is the following result.

\begin{prop}[Proposition 4.12 in \cite{Bouck}]\label{upshotprop}
Led $\alpha$ be a big class, with divisorial Zariski decomposition $\alpha=P(\alpha)+N(\alpha)$. Then $\mathrm{vol}(\alpha)=\mathrm{vol}(P(\alpha))=\int_{X}P(\alpha)^{2n}=c_X(q_X(P(\alpha))^n$.
\end{prop}

Now item 2 of Theorem \ref{mainthm1} easily follows from the last proposition.
\bigskip
\begin{proof1}
We already know from Proposition \ref{volumeofdivisors} and \ref{upshotprop} that, for any big $\mathbb{R}$-divisor $D$ with divisorial Zariski decomposition $D=P(D)+N(D)$, $\mathrm{vol}(D)=(P(D))^{2n}$. Let $\Sigma_P$ be a Boucksom-Zariski chamber, and set $\mathrm{Null}_{q_X}(P)=\{D_1,\dots,D_k\}$. We know that the elements of $\mathrm{Null}_{q_X}(P)$ are linearly independent in $N^1(X)_{\mathbb{R}}$. Hence we can choose a basis $\mathcal{B}$ of $N^1(X)_{\mathbb{R}}$ of the form
\[
\mathcal{B}=\{D_1,\dots,D_k,B_{k+1},\dots,B_{\rho}\},
\] 
where $\rho$ is the Picard number of $X$ and $\{B_{k+1},\dots,B_{\rho}\}$ is a basis for $\mathrm{Null}_{q_X}(P)^{\perp}$. Let $D=P(D)+N(D)=\sum_ix_iB_i+N(D)$ be any element of $\Sigma_P$. We have
\[
\mathrm{vol}(D)=\mathrm{vol}(P(D))=(P(D))^{2n},
\]
which is a homogeneous polynomial of degree $2n$ in the variables $x_i$, and we are done.
\end{proof1}

\section{An example and a remark on the stability chambers}\label{section7}
Let us start this section with the following example.
\begin{ex}[Boucksom-Zariski chambers on $\mathrm{Hilb}^2(S)$]\label{exchambers}
This example is based on the following result, which has been taken from the article \cite{Ulrike} of Ulrike Riess.
\begin{prop}[Lemma 3.2 of \cite{Ulrike}]\label{propulrike}
Let $S$ be a K3 surface with $\mathrm{Pic}(S) \cong \mathbb{Z}\cdot A_S$, for an ample line bundle with $(A_S)^2=2$. Consider the irreducible symplectic manifold $Y:=\mathrm{Hilb}^2(S)$, with the usual decomposition $\mathrm{Pic}(Y)\cong \mathbb{Z}\cdot A \oplus \mathbb{Z} \cdot \delta$, where $A$ is the line bundle associated to $A_S$. Then 
\begin{enumerate}
\item $\mathscr{C}_Y \cap \mathrm{N}^1(Y)_{\mathbb{R}}=\langle A+\delta,A-\delta\rangle$,
\item $\overline{\mathscr{BK}_Y} \cap \mathrm{N}^1(Y)_{\mathbb{R}}=\langle A,A-\delta\rangle $,
\item $\mathrm{Nef}(Y)=\langle A,3A-2\delta\rangle \subseteq N^1(Y)_{\mathbb{R}}$,
\item there is a unique other birational model $Y'$ of $Y$ which is an irreducible symplectic manifold. This satisfies
\[
\mathrm{Nef}(Y')=\langle 3A'-2\delta',A'-\delta' \rangle \subseteq N^1(Y')_{\mathbb{R}},
\]
where $A',\delta' \in \mathrm{Pic}(Y')$ are the line bundles which correspond to $A$ and $\delta$ via the birational transform.
\end{enumerate}
\end{prop}

 Let $Y$ be the IHS manifold of Proposition \ref{propulrike}. We only have one $q_Y$-exceptional prime divisor, namely $2\delta$. It follows from Corollary \ref{structurepseff} and item $1$ of Proposition \ref{propulrike} that $\overline{\mathrm{Eff}(Y)}=\langle A-\delta,\delta \rangle$. We have two Boucksom-Zariski chambers. The first is $\Sigma_M$, where $M$ is any element of $\mathrm{int}\left(\mathrm{Mov}(Y)\right)$, the second is the chamber associated (for example) to the big and $q_Y$-nef divisor $A$, namely $\Sigma_A$. It follows from Proposition \ref{propulrike} that $\overline{\Sigma_M}=\mathrm{Nef}_{q_Y}(Y)=\langle A-\delta,A \rangle$, moreover $\overline{\Sigma_A}=\langle A, \delta \rangle$. We observe that the chamber $\Sigma_M$ is neither closed nor open, while $\Sigma_A$ is open.
 
 With the notation of Proposition \ref{propulrike}, we obtain the following picture, which has been drawn starting from Figure 1 in \cite{Ulrike}.

\begin{figure}[htbp]\label{fig:CONES}
  \centering

\definecolor{light-gray}{gray}{0.95}
  \begin{tikzpicture}[scale=0.9]

       \path[fill=light-gray] (0,0) -- (-4,4) -- (-4,5) -- (4,5) -- (4,0) -- cycle;

       \begin{scope}
         \clip (0,0) -- (-3.33,5) -- (0,5) -- cycle;
      \foreach \x in {0,1,...,20} 
      {\draw [gray] (0-0.2*\x,0+0.2*\x) -- (2-0.2*\x,5+0.2*\x);}
       \path[fill=light-gray] (-2.1,4.3) rectangle (-0.5,4.7);
       \end{scope}

       \begin{scope}
         \clip (0,0) -- (-3.33,5) -- (-4,5) -- (-4,4) -- cycle;
      \foreach \x in {0,1,...,30} 
      {\draw [gray] (0-0.2*\x,0+0.2*\x) -- (4-0.2*\x,2+0.2*\x);}
       \path[fill=light-gray] (-4,4.2) rectangle (-3,4.8);
       \end{scope}

     \foreach \x in {-4,-3,...,4} 
       \foreach \y in {-1,...,5}{
         \fill(\x,\y) circle (1pt);}
       \fill(0,0) circle (2 pt);
       \node (o) at (-0.3,-0.3) {$0$};
       \fill(1,0) circle (2 pt);
       \node (d) at (1,-0.3) {$\delta$};
       \fill(0,1) circle (2 pt);
       \node (h) at (0.3,1) {$H$};
       \fill(-1,1) circle (2 pt);
       \node (l) at (-1.3,0.5) {$H-\delta$};

\draw [line width=1.5](0,0) -- (0,5);
\draw [line width=1.5pt] (0,0) -- (-4,4);
\draw [line width=1.5pt]  (0,0) -- (4,0);
\path[fill=light-gray] (-0.2,2.3) rectangle (0.9,2.8);
\draw [dashed] (0,0) -- (-3.33,5);

       \node (pseffX) at (0.3,2.5) {$\overline{\mathrm{Eff}(Y)}$};
       \node (nefX) at (-1.3,4.5) {$\mathrm{Nef}(Y)$};
       \node (nefX') at (-3.9,4.5) {$\mathrm{Nef}(Y')$};
       \node (sigmaH) at (2.3,2.3) {$\Sigma_{H}$};
       \path[fill=light-gray] (-2.2,2.4) rectangle (-1.4,2.7);
       \node (sigmaM) at (-1.8,2.6) {$\Sigma_M$};

  \end{tikzpicture}
  \caption{Boucksom-Zariski chambers on $\mathrm{Hilb}^2(S)$}
  \label{fig:cones}
\end{figure}
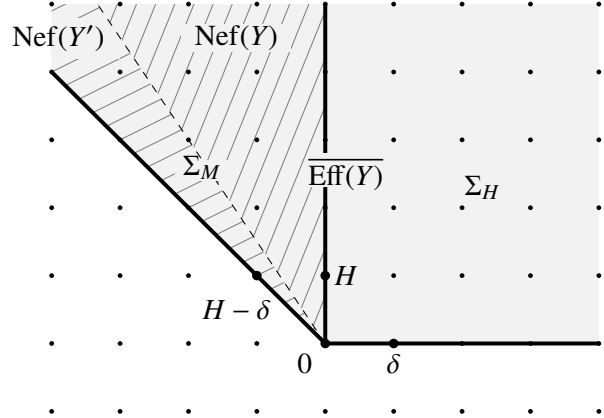
 
 Now we describe the volume function $\mathrm{vol}(-)$ associated with $Y$. We start by observing that in this case, we have $c_Y=3$ (see for example the table on page 78 in \cite{Rap}). By Proposition \ref{upshotprop} and by the proof of item 2 in Theorem \ref{mainthm1}, we have the following description of the volume function.
\begin{enumerate}
\item \textbf{vol(-) on $\mathbf{\Sigma_M}$:} let $[D]\in \Sigma_M$ and write $[D]=x[A]+z[A-\delta]$. Then
\[
\mathrm{vol}(D)=[D]^4=3 \cdot \left[q_Y(x[A]+z[A-\delta])\right]^2=12x^4+48x^2z^2+48x^3z.
\]
\item \textbf{vol(-) on $\mathbf{\Sigma_A}$:} let $[D]\in \Sigma_A$ and write $[D]=x[A]+z[\delta]$. Then
\[
\mathrm{vol}(D)=[P(D)]^4=3 \cdot \left[q_Y(x[A])\right]^2=3x^4(q_Y(A))^2=12x^4.
\]
\end{enumerate}
Outside the big cone $\mathrm{vol}(-)$ is zero and this concludes its description.
\end{ex}
We conclude this section with the following remark about the stability chambers (see Definition \ref{defnstabchamb}) on projective IHS manifolds.
\begin{rmk}[A remark on the stability chambers]\label{rmkstabchambers}
Let $S$ be a smooth projective surface (over any algebraically closed field $\mathbb{K}$) and $Y$ a normal complex projective variety. As we said in the introduction, in \cite{Bau} the authors also studied how "stable base loci" of divisors vary in $N^1(S)_{\mathbb{R}}$, obtaining in this way another decomposition of $\mathrm{Big}(S)$ into chambers, known as \textit{stability chambers}. For further details, we refer the interested reader to \cite{Ein} for the needed positivity notions, and to \cite[Section 2]{Bau} for the part concerning the stability chambers. 
\begin{defn}
The \textit{stable base locus} of an integral divisor $D$ on $Y$ is defined as
\[
\mathbf{B}(D):=\bigcap_{m \geq 1}\mathrm{Bs}(mD)_{\mathrm{red}},
\]
where $\mathrm{Bs}(mD)_{\mathrm{red}}$ is the base locus of the divisor $mD$, considered as a reduced subset of $Y$.
\end{defn}
Bauer, K\"uronya, and Szemberg study a slightly modified version of the stable base locus, namely the augmented base locus. The advantage of working with the latter is that it can be safely studied in the N\'{e}ron-Severi space, namely two numerically equivalent divisors have the same augmented base locus.
\begin{defn}
The augmented base locus of an $\mathbb{R}$-divisor $D$ on $Y$ is defined as
\[
\mathbf{B}_{+}(D):=\bigcap_{D=A+E}\mathrm{Supp}(E),
\]
where the intersection is taken over all the decompositions of $D=A+E$, where $A$ and $E$ are $\mathbb{R}$-divisors such that $A$ is an ample and $E$ is effective.
\end{defn}

\begin{defn}
The restricted base locus of an $\mathbb{R}$-divisor $D$ on $Y$ is defined as
\[
\mathbf{B}_{-}(D):=\bigcup_{A}\mathbf{B}(D+A),
\]
where the union is taken over all ample divisors $A$, such that $D + A$ is a $\mathbb{Q}$-divisor .
\end{defn}

\begin{defn}
Let $D$ be an $\mathbb{R}$-dvisior on $Y$. We say that $D$ is stable if $\mathbf{B}_{-}(D)=\mathbf{B}_{+}(D)$.
\end{defn}
\begin{rmk}
Note that both the augmented base locus and the restricted base locus of a given $\mathbb{R}$-divisor depend only on the numerical equivalence class, so the stability notion does.
\end{rmk}

\begin{defn}\label{defnstabchamb}
Let $D$ be a stable $\mathbb{R}$-divisor on $Y$. The chamber of stability of $D$ is defined as
\[
\mathrm{SC}(D):=\left\{D' \in \mathrm{Big}(Y) \; | \; \mathbf{B}_{+}(D')=\mathbf{B}_{+}(D)\right\}.
\]
\end{defn}

The relevant result obtained in \cite{Bau} is the following.
 \begin{thm}\label{stabilitychambers}
 Let $D$ be a stable big $\mathbb{R}$-divisor on $S$, and $\mathrm{SC}(D)$ its stability chamber. Then
 \[
 \mathrm{int}\left(\mathrm{SC}(D)\right) = \mathrm{int}\left(\Sigma_{P(D)}\right),
 \]
 where $\Sigma_{P(D)}$ is the Zariski chamber associated to the big and nef $\mathbb{R}$-divisor $P(D)$.
 \end{thm}
Theorem \ref{stabilitychambers} tells us that on a smooth projective surface, the Zariski chambers and the stability chambers essentially coincide. It is natural to ask the following question.

\bigskip
\textbf{Question $\blacklozenge$. } Let $D$ be a stable big $\mathbb{R}$-divisor on a projective IHS manifold $X$, and $\mathrm{SC}(D)$ its stability chamber. Is it true that
 \[
 \mathrm{int}\left(\mathrm{SC}(D)\right) = \mathrm{int}\left(\Sigma_{P(D)}\right),
 \]
 where $\Sigma_{P(D)}$ is the Boucksom-Zariski chamber associated to the big and $q_X$-nef $\mathbb{R}$-divisor $P(D)$?
 
\bigskip
 We immediately realized that this is not possible, because the geometry of a higher dimensional projective IHS manifold is more intricate. We now explain why Question $\blacklozenge$ cannot in general be answered positively.
\begin{rmk}
Let $A$ be any ample divisor on $Y$. One can show that \;$\mathrm{SC}(A)=\mathrm{Amp}(Y)$ (see for example Proposition 1.5 and Example 1.7 of \cite{Ein}).
\end{rmk}
Let $A$ be an ample divisor on a projective IHS manifold $X$. If Question $\blacklozenge$ has an affirmative answer, then 
\[
\mathrm{Amp}(X)=\mathrm{SC}(A)=\mathrm{int}(\Sigma_{A}).
\]
But $\mathrm{int}(\Sigma_{A})=\mathrm{int}\left(\mathrm{Mov}(X)\right)$, by Example \ref{exchamber}. If $X$ is the IHS manifold of Proposition \ref{propulrike}, looking at Figure \ref{fig:CONES}, we have $\mathrm{Amp}(X) \subsetneq \mathrm{int}\left(\mathrm{Mov}(X)\right)$, thus Question $\blacklozenge$ cannot be answered positively. This means that, in general, the augmented base loci are not constant in the interior of the Boucksom-Zariski chambers.

\end{rmk}

\section{Simple Weyl chambers}
In this section, we introduce the decomposition of the big cone of a projective IHS manifold into \textit{simple Weyl chambers}. Similarly to what was done for surfaces, we compare this decomposition to that in Boucksom-Zariski chambers. Also, we prove Theorem \ref{mainthm3}. 
 
\begin{defn}
The \textit{simple Weyl chambers} on the projective IHS manifold $X$ are defined as the connected components of the set 
\[
\mathrm{Big}(X) \setminus \left(\bigcup_{D \in \mathrm{Neg}(X)} D^{\perp} \right).
\]
\end{defn}
In this way, we obtain the decomposition of $\mathrm{Big}(X)$ into simple Weyl chambers. The above definition does not give much information about the structure of the simple Weyl chambers, so it is natural to look for a concrete description of them. To begin with, consider a $q_X$-exceptional block $S$ (see Definition \ref{qexcblock}) and define 
\[
W_S:=\left\{\alpha \in \mathrm{Big}(X) \; | \; q_X(\alpha,D) < 0 \mathrm{\; for\; any\; } D \in S \mathrm{ ,\; and\; } q_X(\alpha,D)>0 \mathrm{ \;for \;any \;} D \in \mathrm{Neg}(X) \setminus S \right\}.
\]
We observe that $W_S$ is a convex subset of the N\'{e}ron-Severi, and so is connected.
\begin{lem}\label{wcarenotempty} The following assertions hold true.
\begin{enumerate}
\item The sets $W_S$ are non-empty and open in $N^1(X)_{\mathbb{R}}$.
\item If $S,S'$ are two $q_X$-exceptional blocks, then $W_S \cap W_{S'}=\emptyset$ if and only if $S' \neq S$.
\end{enumerate}
\begin{proof}
\begin{enumerate}
\item If $S=\{\emptyset\}$, then $W_S=\mathrm{int}\left(\mathrm{Mov}(X)\right)$ and we are done. Assume then $S=\{D_1,\dots,D_k\}$. We first prove that $W_S$ is not empty. Indeed, consider the divisor $D=M+\sum_{i=1}^kx_iD_i$ in the variables $x_i$, where $M$ is any divisor lying in $\mathrm{int}\left(\mathrm{Mov}(X)\right)$. We have to find values for the $x_i$ such that the divisor $D$ lies in $W_S$, i.e.  $q_X(D,D_i)<0$ for any $i=1,\dots,k$, $q_X(D,D')>0$ for any $D' \in \mathrm{Neg}(X) \setminus S$, and $D \in \mathrm{Big}(X)$. Let $(a_1,\dots,a_k)$ be any $k$-tuple of negative real numbers, and consider the linear system $\mathcal{S}$ defined by the equations
\[
q_X(M,D_j)+\sum_iq_X(D_i,D_j)x_i=a_j \;,\; j=1,\dots,k.
\]
It follows that
\[
G^{-1} \cdot \left(q_X(M,D_1)-a_1,\dots,q_X(M,D_k)-a_k\right)^{T}=(-x_1,\dots,-x_k)^T,
\]
where $G$ is the Gram matrix of $S$, which is negative-definite and so invertible. Arguing as in Lemma \ref{lemexcepblock} we obtain that the solution $(\lambda_1,\dots,\lambda_k)$ of $\mathcal{S}$ is made of positive numbers. This implies that $D=M+\sum_{i=1}^k\lambda_iD_i$ is big and belongs to $W_S$.

To prove that $W_S$ is open, as $N^1(X)_{\mathbb{R}}$ is finite-dimensional, it suffices to show that, for any $D \in W_S$, $D\pm \epsilon E$ belongs to $W_S$ for $\epsilon$ small enough, where $E$ any $\mathbb{R}$-divisor. As the big cone is open we can assume that $D\pm \epsilon E$ is big. Without loss of generality we can assume $q_X(E,D_i)\geq 0$ for any $i=1,\dots,k$, so that  $q_X(D - \epsilon E,D_i)<0$. Also, choosing $\epsilon$ small enough we have
\[
q_X(D + \epsilon E,D_i)=q_X(D,D_i) + \epsilon q_X(E,D_i)<0, \; \mathrm{for \; any\; i=1,\dots,k},
\]
 thus $q_X(D \pm \epsilon E,D_i)<0$. It remains to check the intersection of $D \pm \epsilon E$ with any $D' \in \mathrm{Neg}(X) \setminus S$. We observe that we can have $q_X(D\pm\epsilon E,D') \leq 0$ only for a finite number of $D' \in \mathrm{Neg}(X) \setminus S$. Indeed, as $D\pm\epsilon E$ is big, we have $D\pm\epsilon E \equiv_{\mathrm{num}} A+N$, where $A$ is ample and $N$ is effective (both depending on the sign between $D$ and $\epsilon E$). Hence 
 \[
 q_X(D\pm \epsilon E,D')=q_X(A+N,D')\leq 0
 \]
 implies that $D'$ is an irreducible component of $N$, and so we can have only a finite number of such $D'$. But then, up to taking $\epsilon$ even smaller, we clearly have
 \[
 q_X(D\pm \epsilon E,D')>0 
 \]
 for any $D' \in \mathrm{Neg}(X) \setminus S$, and we are done.
 \item  The arrow $\Rightarrow$ is obvious. Let us prove $\Leftarrow$. Without loss of generality, we can assume $S' \setminus S \neq \emptyset$. Let $D' \in S' \setminus S$. If $D \in W_S \cap W_{S'}$, then $q_X(D,D')<0$ and $q_X(D,D')>0$, which is clearly absurd.
 \end{enumerate}
\end{proof}
\end{lem}
The following Proposition gives us the desired characterization of the simple Weyl chambers.
\begin{prop}\label{charactwc}
The simple Weyl chambers on $X$ are exactly the sets $\{W_S\}_S$, where $S$ varies among all the $q_X$-exceptional blocks.
\begin{proof}
Let $W$ be a simple Weyl chamber and pick any $D \in W$. We define 
\[
S_D:=\left\{E \in \mathrm{Neg}(X) \; | \; q_X(D,E)<0 \right\},
\]
and observe that $S_D \subset \mathrm{Neg}_{q_X}(D)$, and so $S_D$ is a $q_X$-exceptional block. We now prove that $S:=S_D$ is independent from the element of $W$ we have chosen. Let $D' \neq D$ be another element of $W$. By contradiction, if $S_D\neq S_{D'}$, without loss of generality, we may pick an element $E \in S_D \setminus S_{D'}$, which means $q_X(D,E)<0$ and $q_X(D',E)>0$. Consider the linear functional 
\[
q_X(-,E)\colon N^1(X)_{\mathbb{R}}\to \mathbb{R}, \mathrm{\;such \; that} \; q_X(-,E)(D''):=q_X(D'',E).
\]
As $W$ is connected and $q_X(-,E)$ is continuous, there exists $D'' \in W$ such that $q_X(-,E)(D'')=q_X(D'',E)=0$, hence we reached a contradiction, because $W \subset \mathrm{Big}(X) \setminus \left(\bigcup_{D \in \mathrm{Neg}(X)} D^{\perp}\right)$.  We conclude that $W \subset W_S$. Moreover, we must have $W=W_S$,  because $W$ is a connected component of $\mathrm{Big}(X) \setminus \left(\bigcup_{D \in \mathrm{Neg}(X)} D^{\perp} \right)$ and $W_S \subset \mathrm{Big}(X) \setminus \left(\bigcup_{D \in \mathrm{Neg}(X)} D^{\perp} \right)$ is also connected. To conclude, let $S$ be any $q_X$-exceptional block. Since $W_S$ is connected, then is contained in a simple Weyl chamber $W$. But $W=W_{S'}$ for some $q_X$-exceptional block $S'$, and $W_S$ is not empty by item (1) of Lemma \ref{wcarenotempty}. It follows by item (2) of Lemma \ref{wcarenotempty} that $S=S'$, and so $W_S=W_{S'}$.
\end{proof}
\end{prop}
\begin{cor}\label{otowc}
There is a one-to-one correspondence between the $q_X$-exceptional blocks and the simple Weyl chambers on $X$.
\begin{proof}
This is clear from Proposition \ref{charactwc}.
\end{proof}
\end{cor}

\begin{ex}\label{exwc}
The easiest example of a simple Weyl chamber is that associated with $\emptyset$, i.e.
\[
W_{\emptyset}=\mathrm{int}\left(\mathrm{Mov}(X)\right).
\] 
Indeed 
\[
\mathrm{int}\left(\mathrm{Mov}(X)\right)=\{\alpha \in N^1(X)_{\mathbb{R}}\cap \mathscr{C}_X \; | \;  q_X(\alpha,E)>0 \mathrm{\; for \; any \;} E \in\mathrm{Neg}(X)\}.
\]
\end{ex}
We now compare the decomposition of $\mathrm{Big}(X)$ into Boucksom-Zariski chambers with the one into simple Weyl chambers. Let us start with the following definition.
\begin{defn}
We say that the Boucksom-Zariski chambers on $X$ are \textit{numerically determined} if their interior coincides with the simple Weyl chambers.
\end{defn}

 Since by Corollary \ref{otobzc} there is a one-to-one correspondence between the Boucksom-Zariski chambers and the $q_X$-exceptional blocks, from now on we will indicate the Boucksom-Zariski chamber associated to a $q_X$-exceptional block $S$ with $BZ_S$. We start with the following relevant result, which tells us that even though the two decompositions can differ, already in dimension $2$ (see \cite[Examples]{Bauer1}), the number of Boucksom-Zariski chambers is always the same to that of simple Weyl chambers (this happens also in the case of surfaces, see Theorem 2.5 in \cite{Hanumanthu}).

\begin{prop}\label{cardwcbcsame}
There is a one-to-one correspondence between the simple Weyl chambers and the Boucksom-Zariski chambers on $X$.
\begin{proof}
This is obvious from  Corollary \ref{otowc} and Corollary \ref{otobzc}.
\end{proof}
\end{prop}

The next result, which is a generalization of Theorem 3.4 in \cite{Hanumanthu}, is a criterion for deciding whether a simple Weyl chamber $W_S$ is contained in the Boucksom-Zariski chamber $BZ_S$.

\begin{prop}\label{wcinbzc}
Let $S$ be a $q_X$-exceptional block. Then $W_S \subseteq BZ_S$ is and only if the following condition holds: if $D \in \mathrm{Neg}(X) \setminus S$ and $S \cup \{D\}$ is a $q_X$-exceptional block, then $q_X(D,D')=0$ for any $D' \in S$.
\begin{proof}
Note that we always have 
\[
W_{\emptyset}=\mathrm{int}\left(\mathrm{Mov}(X)\right)\subseteq \mathrm{Big}(X) \cap \mathrm{Nef}_{q_X}(X)= BZ_{\emptyset},
\] 
so we can assume $S \neq \emptyset$. We first prove $\Leftarrow$. If we show that $\mathrm{Neg}_{q_X}(D)=S$ for any $D \in W_S$, by definition of Boucksom-Zariski chamber we are done. Let $S=\{D_1,\dots,D_k\}$ and $D \in W_S$. As usual, let $D=P(D)+N(D)$ be the divisorial Zariski decomposition of $D$. As by hypothesis $q_X(D,D_i)<0$, for any $i=1,\dots,k$, and $P(D)$ is $q_X$-nef, it follows that any $D_i$ is an irreducible component of $N(D)$. This proves $S \subseteq \mathrm{Neg}_{q_X}(D)$. Now we have to prove $\mathrm{Neg}_{q_X}(D) \subseteq S$. By contradiction, suppose 
\[
\mathrm{Neg}_{q_X}(D) \setminus S=\{D'_1,\dots,D'_r\}.
\]
We thus can write $N(D)=\sum_{i=1}^ra'_iD'_i+\sum_{j=1}^ka_jD_j$. By hypothesis, for any $i=1,\dots,r$ and $j=1,\dots,k$, we have $q_X(D'_i,D_j)=0$. Since $D \in W_S$, by definition, we obtain 
\begin{equation}\label{eqnwc}
q_X(D,D'_j)=q_X(N(D),D'_j)=\sum_{i=1}^ra'_iq_X(D'_i,D'_j)=:b_j>0 \;,\; \mathrm{for \; any \; j=1,\dots,r}.
\end{equation}
If we denote by $G$ the intersection matrix of $\{D'_1,\dots,D'_r\}$, the conditions contained in (\ref{eqnwc}) can be resumed by the linear system
\[
(a'_1,\dots,a'_r)G=(b_1,\dots,b_r),
\]
which in turn gives
\[
(a'_1,\dots,a'_r)=(b_1,\dots,b_r)G^{-1},
\]
because $G$ is invertible, as it is negative-definite. By Lemma 4.1 of \cite{Bau} the entries of $G^{-1}$ are non-positive. It follows that $a'_i<0$ for any $i=1,\dots,r$, and this contradicts the effectivity of $N(D)$. Then $\mathrm{Neg}_{q_X}\subseteq S$, and so $\mathrm{Neg}_{q_X} = S$.

Now we prove the arrow $\Rightarrow$. By contradiction, assume that $W_S \subseteq BZ_S$, but $S=\{D_1,\dots,D_k\}$ does not satisfy the condition in the statement, so that there exists $E \in \mathrm{Neg}(X)\setminus S$ such that $S\cup \{E\}$ is a $q_X$-exceptional block, and $q_X(E,D_i)>0$ for some $D_i \in S$.
We will show $W_S \not \subseteq BZ_S$ by constructing a divisor $D$ belonging to $W_S\setminus BZ_S$.
 By  Corollary \ref{otobzc} we can consider the Boucksom-Zariski chamber $BZ_{S\cup \{E\}}$ and we can pick an element $D' \in \mathrm{int}\left(BZ_{S \cup \{E\}}\right)\neq \emptyset$. Note that, by Corollary \ref{interiorchamber}, $D'$ is such that $\mathrm{Neg}_{q_X}(D')=\mathrm{Null}_{q_X}(P(D'))=S\cup \{E\}$. Now, we will construct a big divisor  $D''$ satisfying the following conditions: 
\begin{itemize}
\item $\mathrm{Supp}\left(N(D'')\right)=\left(\cup_{i=1}^kD_i\right) \cup E$.
\item $q_X(D'',E)=q_X(N(D''),E)<0$ and $q_X(D'',D_i)=q_X(N(D''),D_i)<0$, for all $i=1,\dots,k$.
\end{itemize}
Consider the divisor $xE+\sum_{i=1}^kx_iD_i$ in the variables $x,x_1,\dots,x_k$. Let $(b,b_1,\dots,b_r)$ be any $(r+1)$-tuple of negative numbers and consider the linear system $\mathcal{S}$ defined by the equations below.
\begin{align*}
q_X(E,E)x+\sum_{i=1}^kq_X(D_i,E)x_i&=b, \\
q_X(E,D_j)x+\sum_{i=1}^kq_X(D_i,D_j)x_i&=b_i \;,i=1,\dots,k.
\end{align*}
As $S \cup \{E\}$ is a $q_X$-exceptional block, its Gram matrix $G$ is negative-definite, and the entries of $G^{-1}$ are non-positive by Lemma 4.1 of \cite{Bau}. It follows that the unique solution $(c,c_1,\dots,c_k)$ of $\mathcal{S}$ is made of positive numbers. Then the desired divisor $D''$ can be defined as $D'':=P(D')+cE+\sum_{i=1}^kc_iD_i$. Clearly $P(D'')=P(D')$, and $N(D'')$ is the remaining part in the definition of $D''$.
Now, consider the divisor 
\begin{equation}\label{eqnwc9}
D:=P(D'')+\frac{\mathrm{min}\{c,c_1,\dots,c_k\}}{2|q_X(E)|}E+\sum_{i=1}^kc_iD_i=P(D'')+N(D).
\end{equation}
Note that $D$ is big, and its divisorial Zariski decomposition is given by (\ref{eqnwc9}). Clearly $D \notin BZ_S$, because $\mathrm{Neg}_{q_X}(D) \supsetneq S$. To conclude, we show that $D \in W_S$ by computing explicitly $q_X(D,E')$, where $E'$ is any element of $\mathrm{Neg}(X)$.
Let $D_i \in S$ be any element of $S$. Then 
\[
q_X(D,D_i)=\frac{\mathrm{min}\{c,c_1,\dots,c_k\}}{2|q_X(E)|}q_X(E,D_i)+\sum_{j=1}^kc_iq_X(D_j,D_i) \leq q_X(N(D''),D_i)<0.
\]
Intersecting $D$ with $E$ we obtain
\begin{align}\label{eqnwc1}
\begin{split}
q_X(D,E)=\frac{\mathrm{min}\{c,c_1,\dots,c_k\}}{2|q_X(E)|}q_X(E,E)+&\sum_{j=1}^kc_iq_X(D_j,D_i)=\cdots \\
&\cdots=-\frac{1}{2}\mathrm{min}\{c,c_1,\dots,c_k\}+\sum_{j=1}^kc_iq_X(D_j,E)>0,
\end{split}
\end{align}
and the last inequality in (\ref{eqnwc1}) is true because, by assumption, there exists at least one $j\in\{1,\dots,k\}$ satisfying $q_X(D_j,E)>0$.
Intersecting $D$ with any $E'\in \mathrm{Neg}(X) \setminus\left(S \cup \{E\}\right)$ we obtain
\begin{align}\label{eqnwc2}
\begin{split}
q_X(D,E')=q_X(P(D''),E')+\frac{\mathrm{min}\{c,c_1,\dots,c_k\}}{2|q_X(E)|}& q_X(E,E')+ \cdots \\ \cdots +\sum_{j=1}^kc_iq_X(D_j,E') &\geq q_X(P(D''),E')>0.
\end{split}
\end{align}
The first inequality in (\ref{eqnwc2}) is true because $q_X$ is an intersection product, while the second  because $P(D'')=P(D')$ is $q_X$-nef, and $\mathrm{Null}_{q_X}(P(D'))=S \cup \{E\}$. With the last computation, we reached a contradiction, thus we are done.
\end{proof}
\end{prop}
 Now we do the opposite, i.e. we exhibit a criterion explaining when the interior of a Boucksom-Zariski chamber $BZ_S$ is contained in the simple Weyl chamber $W_S$.
 \begin{prop}\label{bzcinwc}
 Let $S$ be a $q_X$-exceptional block. Then $\mathrm{int}\left(BZ_S\right)\subset W_S$ if and only if the $q_X$-exceptional prime divisors belonging to $S$ are pairwise $q_X$-orthogonal. 
 \begin{proof}
 We first prove $\Leftarrow$. Set $S=\{D_1,\dots,D_k\}$. Let $D \in \mathrm{int}\left(BZ_S\right)$, with divisorial Zariski decomposition $D=P(D)+N(D)$, where $N(D)=\sum_{i=1}^ka_iD_i$. By hypothesis we have $q_X(D_i,D_j)=0$ for any $i\neq j$, thus 
 \[
 q_X(D,D_i)=q_X(N(D),D_i)=a_iq_X(D_i)<0,
 \] 
 for any $D_i \in S$. Now, pick a divisor $E \in \mathrm{Neg}(X)\setminus S$. Then $q_X(D,E)\geq 0$, because $q_X$ is an intersection product and $P(D)$ is $q_X$-nef. If $q_X(D,E)=0$, then $q_X(P(D),E)=0$, and so $S=\mathrm{Neg}_{q_X}(D)\neq \mathrm{Null}_{q_X}(P(D))$. But this is absurd, because $D\in \mathrm{int}\left(BZ_S\right)$ if and only if $\mathrm{Neg}_{q_X}(D)= \mathrm{Null}_{q_X}(P(D))$, by Corollary \ref{interiorchamber}. Hence $q_X(D,E)>0$, and $\mathrm{int}\left(BZ_S\right)\subset W_S$. We now show $\Rightarrow$. By contradiction, assume that $\mathrm{int}\left(BZ_S\right)\subset W_S$, and there exist two elements of $S$ which are not $q_X$-orthogonal, say $D_1,D_2$. Let $M$ be any divisor lying in $\mathrm{int}\left(\mathrm{Mov}(X)\right)$. By Lemma \ref{lemexcepblock} there exists a $k$-tuple of positive numbers $(\lambda_1,\dots,\lambda_k)$ such that $M+\sum_{i=1}^k\lambda_iD_i$ is $q_X$-nef with $\mathrm{Null}_{q_X}(M+\sum_{i=1}^k\lambda_iD_i)=S$. Now, choose a $k$-tuple of positive numbers $(\mu_1,\dots,\mu_k)$ such that $\mu_i > \lambda_i$ for any $i=1,\dots,k$. Finally, pick a real number $\eta$ satisfying $0<\eta <\frac{\mu_2-\lambda_2}{|q_X(D_1)|}$ and consider the divisor
 \[
 D=M+(\lambda_1+\eta)D_1+\sum_{i=2}^k\mu_iD_i.
 \]
 It follows by construction that the divisorial Zariski decomposition of $D$ is given by $P(D)=M+\sum_{i=1}^k\lambda_iD_i$, and $N(D)=\eta D_1+\sum_{i=2}^k(\mu_i-\lambda_i)D_i$. By construction $\mathrm{Neg}_{q_X}(D)=\mathrm{Null}_{q_X}(P(D))$, thus $D \in \mathrm{int}(BZ_S)$, by Corollary \ref{interiorchamber}. But
 \begin{equation}\label{eqnbzc3}
 q_X(D,D_1)=\eta q_X(D_1)+\sum_{i=2}^k(\mu_i-\lambda_i)q_X(D_i,D_1)\geq \eta q_X(D_1)+(\mu_2-\lambda_2)>0,
 \end{equation}
 and the last inequality in (\ref{eqnbzc3}) is true by the choice of $\eta$. Thus $D \notin W_S$, which is a contradiction. 
 \end{proof}
 \end{prop}
Before proving Theorem \ref{mainthm3}, we need the following chain of equivalences. 
 \begin{lem}\label{bzcnumericallydetermined0}
 The following assertions are equivalent:
 \begin{enumerate}
 \item If two $q_X$-exceptional prime divisors $D_1,D_2$ intersect properly, i.e. $q_X(D_1,D_2)>0$, then $q_X(D_1,D_2)\geq \sqrt{q_X(D_1)q_X(D_2)}$.
 \item If $D_1,D_2$ are $q_X$-exceptional prime divisors such that $\{D_1,D_2\}$ is a $q_X$-exceptional block, then $q_X(D_1,D_2)=0$.
 \item Let $S$ be a $q_X$-exceptional block. If $E \in \mathrm{Neg}(X)\setminus S$, and $S\cup \{E\}$ is a $q_X$-exceptional block, then $q_X(E,E')=0$ for any $E' \in S$.
 \item Let $S$ be a $q_X$-exceptional block. Then the $q_X$-exceptional prime divisors of $S$ are pairwise $q_X$-orthogonal.
 \end{enumerate}
 \begin{proof}
 We first prove $(1)\Rightarrow (2)$. Suppose that $(2)$ does not hold. Then there exist $q_X$-exceptional prime divisors $D_1,D_2$ such that the block $S=\{D_1,D_2\}$ is $q_X$-exceptional, and $q_X(D_1,D_2)\neq 0$. Then $q_X(D_1,D_2)> 0$, because $q_X$ is an intersection product. As the Gram matrix of $S$ is negative-definite, then
 \begin{equation}\label{eqnwc4}
 x^2q_X(D_1)+y^2q_X(D_2)+2xyq_X(D_1,D_2)<0,
 \end{equation}
 for any couple $(x,y)\neq(0,0)$. As $\sqrt{q_X(D_1)q_X(D_2)} \neq 0$, using the relation (\ref{eqnwc4}), we obtain
 \begin{equation}\label{eqnwc5}
 q_X(D_1,D_2)<\frac{\left[x\sqrt{-q_X(D_1)}\right]^2+\left[y\sqrt{-q_X(D_2)}\right]^2}{2xy\sqrt{q_X(D_1)q_X(D_2)}}\cdot \sqrt{q_X(D_1)q_X(D_2)}.
 \end{equation}
 If we choose $x=\sqrt{-q_X(D_1)}$ and $y=\sqrt{-q_X(D_2)}$, then (\ref{eqnwc5}) gives 
 \[
 q_X(D_1,D_2)<\sqrt{q_X(D_1)q_X(D_2)},
 \] 
 which contradicts item $(1)$ of the statement.
 
 Now we prove $(2)\Rightarrow (1)$. We can assume that the block $\{D_1,D_2\}$ is not $q_X$-exceptional, as otherwise, its Gram matrix would be negative-definite, and so $q_X(D_1,D_2)=0$ by hypothesis, and we are only interested in couples of $q_X$-exceptional prime divisors intersecting properly.  By our assumption, there exists a couple $(x,y)$ of real numbers of the same sign satisfying
 \begin{equation}\label{eqnwc6}
 x^2q_X(D_1)+y^2q_X(D_2)+2xyq_X(D_1,D_2)\geq 0.
 \end{equation}
 Indeed, if such a couple were made of numbers with different signs, then the quantity (\ref{eqnwc6}) would be strictly negative, which is not possible by our assumption.
 As $\sqrt{q_X(D_1)q_X(D_2)} \neq 0$, the above inequality implies
 \begin{align}\label{eqnwc7}
 \begin{split}
  q_X(D_1,D_2)\geq\frac{\left[x\sqrt{-q_X(D_1)}\right]^2+\left[y\sqrt{-q_X(D_2)}\right]^2}{2xy\sqrt{q_X(D_1)q_X(D_2)}} & \cdot \sqrt{q_X(D_1)q_X(D_2)}\geq \cdots \\ 
  & \cdots \geq \sqrt{q_X(D_1)q_X(D_2)},
  \end{split}
 \end{align}
 and we are done.
 The implications $(3)\Rightarrow (2)$ and $(4)\Rightarrow (2)$ are trivial, so we prove $(2)\Rightarrow (3)$. Let $S$ be a $q_X$-exceptional block. If $S\cup \{E'\}$ is also $q_X$-exceptional for some $E' \in \mathrm{Neg}(X) \setminus S$, and $E \in S$, then the block $\{E,E'\}$ is $q_X$-exceptional. By hypothesis we obtain $q_X(E,E')=0$.
 
 To conclude, we prove $(2)\Rightarrow (4)$. Let $S$ be a $q_X$-exceptional block. If $D_1,D_2$ belongs to $S$, then also the block $\{D_1,D_2\}$ is $q_X$-exceptional, and by hypothesis we obtain $q_X(D_1,D_2)=0$. 
 \end{proof}
 \end{lem} 
We are now ready to prove Theorem \ref{mainthm3}, which tells us when the Boucksom-Zariski chambers are numerically determined.
\medskip
  \begin{proof3}
  We first prove $(2)\Rightarrow (1)$. By Lemma \ref{bzcnumericallydetermined0}, item $(2)$ of Theorem \ref{mainthm3} is equivalent to item $(3)$ and $(4)$ of Lemma \ref{bzcnumericallydetermined0}. By Proposition \ref{wcinbzc} and \ref{bzcinwc} we obtain that $W_S\subset BZ_S$ and $\mathrm{int}\left(BZ\right) \subset{W_S}$ for any $q_X$-exceptional block $S$, so that $W_S=\mathrm{int}\left(BZ_S\right)$, and we are done. We now prove $(1)\Rightarrow (2)$ in Theorem \ref{mainthm3}. As by hypothesis $W_S=\mathrm{int}\left(BZ_S\right)$ for any $q_X$-exceptional block $S$,  by Proposition \ref{wcinbzc} we obtain that item $(3)$ of Lemma \ref{bzcnumericallydetermined0} holds true, and, by Lemma \ref{bzcnumericallydetermined0}, this is equivalent to item $(2)$ of Theorem \ref{mainthm3}. This concludes the proof.  
  \end{proof3}
 
  The following criterion, explaining when a Boucksom-Zariski chamber intersects a given simple Weyl chamber, is a generalization of Theorem 3.6 in \cite{Hanumanthu}.
 \begin{prop}\label{intersectionwcbzc}
 Let $S,S'$ be two $q_X$-exceptional blocks. Then $W_{S'} \cap BZ_{S} \neq \emptyset$ if and only if $S' \subseteq S$ and each subset $T \subseteq  S \setminus S'$ satisfies the following property: there exist two divisors, one in $T$ and one in $S\setminus T$, intersecting properly (with respect to $q_X$).
 \begin{proof}
 We first prove $\Rightarrow$. Let $D \in W_{S'} \cap BZ_S$ with divisorial Zariski decomposition $D=P(D)+N(D)$. Then for every $D' \in S'$ we have $q_X(D,D')<0$. Moreover $q_X(N(D),D')<0$, because $P(D)$ is $q_X$-nef, thus $D'$ is an irreducible component of $N(D)$. But $S$ coincides with the set of irreducible components of $N(D)$, hence $S' \subseteq S$. Now, suppose that there exists a subset $T \subseteq S \setminus S'$ not satisfying the condition in the statement. This means that $q_X(E,E')=0$ for any $E \in T$ and $E' \in S\setminus T$ (because $q_X$ is an intersection product). We can write
 \[
 D=P(D)+N(D)=P(D)+\sum_{E_i \in T}a_i E_i +\sum_{E_j \in S \setminus T} b_jE_j.
 \]
 Since $D \in W_{S'}$, and by our assumption, we have
 \[
 q_X(D,E_j)= q_X\left(\sum_{E_i \in T}a_i E_i, E_j\right)>0,
 \] for any $E_j \in T \subset S \setminus S'$. The intersection matrix of $T$ is negative-definite, hence, by Lemma 4.1 of \cite{Bau}, the entries of $T^{-1}$ are non-positive. It follows that the $a_i$ appearing in $\sum_{E_i \in T}a_i E_i$ are negative, and this contradicts the effectivity of $N(D)$. Hence $\Rightarrow$ is proved.
 
 We now prove $\Leftarrow$. First, we show that there exists an element $D' \in W_{S'}\cap BZ_{S'}$.
 We construct $D'$ explicitly and of the form
 \begin{equation}\label{definitionD'}
 D'=M+\sum_{D'_i \in S'}\lambda_iD'_i+\sum_{D'_i\in S'}(\mu_i-\lambda_i)D'_i,
 \end{equation}
 where $M$ is any divisor lying in $\mathrm{int}\left(\mathrm{Mov}(X)\right)$, and satisfying $P(D')=M+\sum_{D'_i \in S'}\lambda_iD'_i$, $N(D')=\sum_{D'_i\in S'}(\mu_i-\lambda_i)D'_i$.
 Arguing as in Lemma \ref{lemexcepblock} we can find the $\lambda_i>0$ in such a way that $\mathrm{Null}_{q_X}(P(D'))=S'$, so that $P(D')$ is $q_X$-nef. Now, let $\{y_i<0 \; | \; 1\leq i \leq \mathrm{card}\left(S'\right)\}$ be a set of negative real numbers, and consider the linear system of equations defined by
 \[
 \sum_{D'_i \in S'}q_X(D'_i,D'_j)x_i=y_i, \; \mathrm{for \; any \; D'_j \in S'}.
 \]
 Again, arguing as in Lemma \ref{lemexcepblock}, we obtain a unique solution for this linear system, made up of positive real numbers. Using the variables $x_i$ to indicate the solution of the linear system above, we define $\mu_i:=\lambda_i+x_i$. Defining $D'$ as in (\ref{definitionD'}), with the $\lambda_i$ and the $\mu_i$ we introduced so far, clearly it holds $D' \in W_{S'}\cap BZ_{S'}$. Now, if $S=S'$ we are done, otherwise $S' \subsetneq S$, and we can consider the set
 \[
 S'':=\{E''\in S \setminus S' \; | \; q_X(E'',E')>0 \mathrm{\; for\; some\; }E'\in S' \}.
 \]
  We note that $S''$ is not empty by hypothesis. We describe how to construct a divisor $D''$ belonging to $W_{S'}\cap BZ_{S' \cup S''}$. We claim that $D''$ can be explicitly constructed of the form
  \[
  D''=M+\sum_{D''_i \in S'}\alpha_iD''_i+\sum_{D''_i \in S''}\beta_iD''_i+nN(D')+\sum_{D_i'' \in S''}D''_i,
  \]
  where 
  \begin{itemize}
  \item $n$ is any positive natural number, and the $\alpha_i, \beta_i$ are positive real numbers,
  \item $M \in \mathrm{int}\left(\mathrm{Mov}(X)\right)$, 
  \item $P(D'')=M+\sum_{D''_i \in S'}\alpha_iD''_i+\sum_{D''_i \in S''}\beta_iD''_i$, 
  \item $N(D'')=nN(D')+\sum_{D_i'' \in S''}D''_i$.
  \end{itemize}
  Indeed, arguing as in Lemma \ref{lemexcepblock}, we can construct $D''$ having the listed properties. Now, if $E\in S'$, then 
  \[
  q_X(D'',E)=nq_X(N(D'),E)+q_X\left(\sum_{D''_i\in S''}D''_i,E\right).
  \] 
  As $D' \in W_{S'}$, we have $q_X(N(D'),E)<0$, and choosing $n$ large enough we will have $q_X(D'',E)<0$ for any $E \in S'$. To conclude, let $E$ be any $q_X$-exceptional prime divisor not belonging to $S'$. If $E \notin S''$, then $q_X(D'',E)>0$, by construction of $D''$. If $E \in S''$, by definition of $S''$ and by assumption, there exists an element $E'$ in $S'$ such that $q_X(E,E')>0$. Now, 
  \[
  q_X(D'',E)=nq_X(N(D'),E)+q_X\left(\sum_{D''_i\in S''}D''_i,E\right),
  \] 
  all the terms of $q_X(N(D'),E)$ are non-negative and at least one of them is positive. Thus, increasing again $n$, if necessary, we obtain $q_X(D'',E)>0$ for any $E \in \mathrm{Neg}(X)\setminus S'$. We have thus constructed $D'' \in W_{S'}\cap BZ_{S' \cup S''}$.
If $S=S'\cup S''$ we are done. Otherwise, we can repeat the above argument. Proceeding like this, since $S$ is finite, we will end up with the desired divisor $D \in W_{S'} \cap BZ_S$.   
 \end{proof}
 \end{prop}
 
 \begin{rmk}
 By the proof of $\Leftarrow$ in Proposition \ref{intersectionwcbzc}, we obtain that $W_S \cap BZ_S \neq \emptyset$, for any $q_X$-exceptional block $S$.
 \end{rmk}
 
  \begin{ack}
This work will be part of my PhD thesis. I would like to express my gratitude to my supervisor, Prof. Gianluca Pacienza, for following me constantly, for sharing with me his knowledge, for his suggestions and hints, and for introducing me to this beautiful part of Algebraic Geometry. Also, I would like to thank Prof. Giovanni Mongardi, for the fruitful conversations and his suggestions and hints. 
\end{ack}

\printbibliography
\end{document}